\documentclass[11pt, a4paper]{article}
\usepackage[textwidth = 7in, textheight=25cm,nohead]{geometry}
\usepackage{graphicx, amssymb, amsmath, amsthm, amstext, booktabs, float, multirow, rotating,natbib}
\usepackage{bm}
\usepackage{mathtools}
\usepackage[mathscr]{euscript}
\usepackage{authblk}
\usepackage{enumerate}
\usepackage{nicefrac}
\usepackage{setspace}
\usepackage{array}

\usepackage{color}

\usepackage[colorlinks=true,linkcolor=blue,citecolor=blue,pdfborder={0 0 0}]{hyperref}
\usepackage{tikz}
\usetikzlibrary{arrows}

\usepackage{listings}
\usepackage{xcolor}

\let\emptyset\varnothing

\numberwithin{equation}{section} 
\numberwithin{figure}{section} 
\numberwithin{table}{section} 
\newtheorem{definition}{Definition}[section] 
\newtheorem{lemma}{Lemma}[section] 
\newtheorem{theorem}{Theorem}[section] 
 
\newtheorem{proposition}{Proposition}[section] 
\newtheorem{corollary}{Corollary}[section]
\newtheorem{condition}{Condition}[section] 

\theoremstyle{remark} 
\newtheorem{remark}{Remark}[section] 

\theoremstyle{definition} 
\newtheorem{example}{Example}[section]

\bibpunct{(}{)}{,}{a}{,}{,}

\makeatletter
\newcommand{\doublewidetilde}[1]{{%
		\mathpalette\double@widetilde{#1}%
}}
\newcommand{\double@widetilde}[2]{%
	\sbox\z@{$\m@th#1\widetilde{#2}$}%
	\ht\z@=.9\ht\z@
	\widetilde{\box\z@}%
}
\makeatother

\makeatletter
\newcommand*\bigcdot{\mathpalette\bigcdot@{.5}}
\newcommand*\bigcdot@[2]{\mathbin{\vcenter{\hbox{\scalebox{#2}{$\m@th#1\bullet$}}}}}
\makeatother

 \let\mathscr\relax
\usepackage[scr]{rsfso}

\DeclareFontFamily{U}{mathx}{}
\DeclareFontShape{U}{mathx}{m}{n}{<-> mathx10}{}
\DeclareSymbolFont{mathx}{U}{mathx}{m}{n}
\DeclareMathAccent{\widehat}{0}{mathx}{"70}
\DeclareMathAccent{\widecheck}{0}{mathx}{"71}

\def\ppn{\vskip 6pt \noindent }

\def\R{{\mathbb{R}}}
\def\N{{\mathbb{N}}}

\def\P{{\mathbb{P}}}

\newcommand{{\Xs}}{{\cal X}}
\newcommand{{\Ys}}{{\cal Y}}
\newcommand{{\Ls}}{{\cal L}}
\newcommand{{\Ss}}{{\cal S}}
\newcommand{{\Ms}}{{\cal M}}
\newcommand{{\Gs}}{{\cal G}}
\newcommand{{\Hs}}{{\cal H}}
\newcommand{{\Ns}}{{\cal N}}
\newcommand{{\Is}}{{\cal I}}
\newcommand{{\Vs}}{{\cal V}}
\newcommand{{\Ds}}{{\cal D}}
\newcommand{{\Bs}}{{\cal B}}
\newcommand{{\Cs}}{{\cal C}}
\newcommand{{\Rs}}{{\cal R}}
\newcommand{{\Es}}{{\cal E}}
\newcommand{{\Fs}}{{\cal F}}
\newcommand{{\Us}}{{\cal U}}
\newcommand{{\Ps}}{{\cal P}}

\newcommand{{\Xss}}{{\mathfrak{X}}}
\newcommand{{\Yss}}{{\mathfrak{Y}}}

\newcommand{{\ttheta}}{{\bm{\theta}}}
\newcommand{{\tthetaXY}}{{\bm{\theta}_{\scriptscriptstyle XY}}}
\newcommand{{\Ttheta}}{{\bm{\Theta}}}
\newcommand{{\Oomega}}{{\bm{\Omega}}}
\newcommand{{\oomega}}{{\bm{\omega}}}
\newcommand{{\mmu}}{{\bm{\mu}}}
\newcommand{{\aalpha}}{{\bm{\alpha}}}
\newcommand{{\bbeta}}{{\bm{\beta}}}
\newcommand{{\ggamma}}{{\bm{\gamma}}}
\newcommand{{\Ssigma}}{{\bm{\Sigma}}}
\newcommand{{\Sss}}{{\bm{\Ss}}}
\newcommand{{\pp}}{{\mathbf p}}
\newcommand{{\qq}}{{\mathbf q}}
\newcommand{{\ww}}{{\mathbf w}}
\newcommand{{\mm}}{{\mathbf m}}
\newcommand{{\ii}}{{\bm i}}
\newcommand{{\jj}}{{\bm j}}
\newcommand{{\uu}}{{\mathbf u}}
\newcommand{{\vv}}{{\mathbf v}}
\newcommand{{\ppi}}{{\bm{\pi}}}
\newcommand{{\phhi}}{{\bm{\phi}}}
\newcommand{{\pssi}}{{\bm{\psi}}}
\newcommand{{\XX}}{{\mathbf X}}
\newcommand{{\UU}}{{\mathbf U}}
\newcommand{{\BB}}{{\mathbf B}}
\newcommand{{\DD}}{{\mathbf D}}
\newcommand{{\KK}}{{\mathbf K}}
\newcommand{{\WW}}{{\mathbf W}}
\newcommand{{\HH}}{{\mathbf H}}
\newcommand{{\II}}{{\mathbf I}}
\newcommand{{\PP}}{{\mathbf P}}
\newcommand{{\yy}}{{\mathbf y}}
\newcommand{{\ee}}{{\mathbf e}}
\newcommand{{\ab}}{{\mathbf a}}

\newcommand{{\dd}}{{\mathbf d}}
\newcommand{{\zero}}{{\mathbf 0}}
\newcommand{{\uno}}{{\mathbf 1}}

\newcommand{{\dep}}{{\mathfrak{D}}}
\newcommand{{\DDelta}}{{\bm \Delta}}
\newcommand{{\pXY}}{{\pp_{{\scriptscriptstyle XY}}}}
\newcommand{{\pXYstar}}{{\pp^*_{{\scriptscriptstyle XY}}}}
\newcommand{{\SXY}}{{\Ss_{{\scriptscriptstyle XY}}}}
\newcommand{{\pX}}{{\pp_{{\scriptscriptstyle {\bm X}}}}}
\newcommand{{\pY}}{{\pp_{{\scriptscriptstyle Y}}}}
\newcommand{{\gammaXY}}{{{\bm{\gamma}}_{{\scriptscriptstyle XY}}}}
\newcommand{{\gammaX}}{{{\bm{\gamma}}_{{\scriptscriptstyle X}}}}
\newcommand{{\gammaY}}{{{\bm{\gamma}}_{{\scriptscriptstyle Y}}}}
\newcommand{{\hatpXY}}{{{\widehat{\pp}_{{\scriptscriptstyle XY};n}}}}
\newcommand{{\hatpX}}{{{\widehat{\pp}_{{\scriptscriptstyle X};n}}}}
\newcommand{{\hatpY}}{{{\widehat{\pp}_{{\scriptscriptstyle Y};n}}}}
\newcommand{{\tildehatpXY}}{{{\widetilde{\widehat{\pp}}_{{\scriptscriptstyle XY};n}}}}
\newcommand{{\hatgammaXY}}{{{\widehat{\ggamma}_{{\scriptscriptstyle XY};n}}}}
\newcommand{{\hatthetaXY}}{{{\widehat{\ttheta}_{{\scriptscriptstyle XY};n}}}}
\newcommand{{\checkpXY}}{{{\widecheck{\pp}_{{\scriptscriptstyle XY};n}}}}
\newcommand{{\tildepXY}}{{{\widetilde{\pp}_{{\scriptscriptstyle XY}}}}}
\newcommand{{\tildegammaXY}}{{{\widetilde{\ggamma}_{{\scriptscriptstyle XY}}}}}
\newcommand{{\tildehatgammaXY}}{{{\widetilde{\widehat{\ggamma}}_{{\scriptscriptstyle XY};n}}}}
\newcommand{\transp}[1]{\ensuremath{#1^{\scriptscriptstyle \top}}}
\newcommand{\vect}{\mathrm{vec}}

\newcommand\indep{\protect\mathpalette{\protect\independenT}{\perp}}
\def\independenT#1#2{\mathrel{\rlap{$#1#2$}\mkern2mu{#1#2}}}

\newcommand{\indic}[1]{
	\hbox{${\it 1}\hskip -4.5pt I_{\{ #1 \}}$}
}

\newcommand{{\toL}}{{\overset{\mathcal{L}}{\longrightarrow}\ }}
\newcommand{{\toas}}{{\ \overset{\text{a.s.}}{\longrightarrow}\ }}

\newcommand{{\dou}}{$\leadsto$\ }

\DeclareMathOperator{\diag}{diag}
\DeclareMathOperator{\Supp}{Supp}

\newcommand{\pkg}[1]{{\normalfont\fontseries{b}\selectfont #1}}
\let\proglang=\textsf
\let\code=\texttt

\newcommand{\un}{\underline}

\lstdefinestyle{Rcode}{
	language=R,
	basicstyle=\ttfamily\footnotesize,
	keywordstyle=\color{blue!60!black},
	commentstyle=\color{green!40!black},
	stringstyle=\color{red!60!black},
	numbers=left,
	numberstyle=\tiny\color{gray},
	stepnumber=1,
	numbersep=6pt,
	frame=single,
	breaklines=true,
	keepspaces=true,
	columns=fullflexible,
	showstringspaces=false
}

\begin{document}

\setlength{\belowdisplayskip}{5pt} \setlength{\belowdisplayshortskip}{3pt}
\setlength{\abovedisplayskip}{5pt} \setlength{\abovedisplayshortskip}{0pt}

\title{The empirical discrete copula process}

\author[1]{\textsc{Gery Geenens}\thanks{Corresponding author: {\tt ggeenens@unsw.edu.au}.}}
\author[2]{\textsc{Ivan Kojadinovic}\thanks{{\tt ivan.kojadinovic@univ-pau.fr}.}}
\author[2,3]{\textsc{Tommaso Martini}\thanks{{\tt tommaso.martini@unito.it}.}}
\affil[1]{School of Mathematics and Statistics, UNSW Sydney, Australia}
\affil[2]{Laboratoire de math\'ematiques et applications, Universit\'e de Pau et des Pays de l'Adour, France}
\affil[3]{{Dipartimento Interateneo di Scienze, Progetto e Politiche del Territorio, Universita degli Studi di Torino, Italy}}

\date{\today}
\maketitle
\thispagestyle{empty} 


\begin{abstract} \noindent This paper develops a general inferential framework for discrete copulas on finite supports in any dimension. The copula of a multivariate discrete distribution is defined as Csiszar's $\Is$‐projection (i.e., the minimum-Kullback–Leibler divergence projection) of its joint probability array onto the polytope of uniform-margins probability arrays of the same size, and its empirical estimator is obtained by applying that same projection to the array of empirical frequencies observed on the sample. Under the assumption of random sampling, strong consistency and $\sqrt{n}$-asymptotic normality of the empirical copula array is established, with an explicit `sandwich' form for its covariance. The theory is illustrated by deriving the large-sample distribution of Yule's concordance coefficient (the natural analogue of Spearman's $\rho_\text{S}$ for bivariate discrete distributions) and by constructing a test for quasi-independence in multivariate contingency tables.  Our results not only complete the foundations of discrete-copula inference but also connect directly to entropically regularised optimal transport and other minimum-divergence problems.
	
\end{abstract}

\section{Introduction}\label{sec:intro}

While the theory and practice of characterising dependence between random variables have long been dominated by Pearson's correlation, it is now well understood that this measure, closely associated with the idealised multivariate Gaussian setting, may be ill suited to real-world data, which are seldom Gaussian. This has renewed interest in modern, robust, and flexible tools for dependence modelling. A major development in this direction has been the growing use of copulas.

\ppn The copula approach stems from representation theorems due to \cite{Sklar59}, which suggest that the dependence structure in any random vector $(X_1,\ldots,X_d)$ may be captured by a $d$-variate function $C_{X_1,\ldots,X_d}$, called {\it copula}, in a way insensitive to the individual behaviour of each $X_k$ ($k=1,\ldots,d$). {\it When it holds}, this property -- referred to as `{\it margin-freeness}' -- allows the construction of flexible multivariate models by separately specifying marginal distributions for the $X_k$'s and the dependence/copula. This is frequently presented as the main benefit of the copula construction; see e.g.\ \citet[Section 1.6]{Joe2015}.

\ppn Sklar's representation has proved well suited when all variables $X_1,\ldots,X_d$ are continuous; in particular, the copula is then unique and margin-free.  However, limitations have long been recognised in non-continuous cases, most notably non-uniqueness and non-identifiability of the copula \citep{Marshall96,Genest07,Faugeras17}. A more fundamental issue is that, once discreteness is present, dependence and marginal behaviour are no longer cleanly disentangled by Sklar's representation \citep{Geenens23short}, making copulas lose their margin-free nature and, with it, their main reason-for-being. Many appealing properties of copula models thus break down outside continuous settings. This affects any model based on Sklar's representation, including those involving continuous extensions through multilinear interpolation \citep{Denuit05,Ruschendorf09,Genest14}.

\ppn To reconcile copulas with the discrete setting, \cite{Geenens2020} argued they should be understood from a more fundamental perspective as the representative of an equivalence class of distributions sharing the same dependence structure. Inspired by ideas of \cite{Yule12}, this led to an alternative copula construction where all pleasant properties (in particular `margin-freeness') carry over to the discrete setting (see Section \ref{sec:bivdiscrcop}). This construction shares with Sklar's its principal strength: enabling the construction of flexible models where marginals and dependence structure can be specified separately, but for \emph{discrete} random vectors. It also provides a basis for defining margin-free measures of dependence and concordance between discrete variables, which aligns perfectly with how \citet{Fisher06} described the statistical benefits of copula modelling: `\textit{Copulas} [are] \textit{of interest to statisticians for two main reasons: first, as a means of studying scale-free measures of dependence; and secondly, as a starting point for constructing families of bivariate distributions}.' Many problems involving continuous random vectors that have benefited from Sklar's formalism (see \citealp{Genest24} for a review) may thus admit natural discrete analogues through \cite{Geenens2020}'s construction.

\ppn A fundamental step towards realising this potential is the development of empirical methods for estimating discrete copulas from data, which is the focus of the present paper. \cite{Geenens2020} laid the foundations but did not investigate inference. That task was taken up in \cite{Kojadinovic24}, who addressed nonparametric, semi-parametric and parametric estimation and goodness-of-fit tests, but in the bivariate case ($d=2$) only and under a possibly restrictive rectangular support condition. Here, we address inference for discrete copula models in any dimension $d \in \N$, under support assumptions as broad as possible. The only significant restriction is that we consider distributions with finite supports -- the infinite support case presents unique challenges to be investigated in a subsequent paper. 

\ppn Our main theoretical contribution is Theorem~\ref{thm:gammahat}, the discrete analogue of weak convergence of the empirical copula process in the classical all-continuous case \citep{Fermanian04,Segers12} -- the theoretical cornerstone of all inferential results for continuous copulas, with the same role expected of Theorem~\ref{thm:gammahat} in discrete settings. The paper is organised as follows. Section~\ref{sec:bivdiscrcop} provides a brief refresher of \cite{Geenens2020}'s discrete copula in the bivariate case, facilitating the transition to the multidimensional construction introduced in Section~\ref{sec:dvarcop}. Section~\ref{sec:empcop} studies a nonparametric estimator of a $d$-dimensional discrete copula and its asymptotic properties, including Theorem~\ref{thm:gammahat}. Section~\ref{sec:ill} illustrates these results through examples: the asymptotic distribution of Yule's coefficient of concordance -- the discrete version of Spearman's $\rho$ \citep[Section 6.6]{Geenens2020} -- is derived in Section~\ref{subsec:Yule}, and a test for `{quasi-independence}' is proposed in Section~\ref{subsec:quasitest}. Section~\ref{sec:ccl} concludes.

\section{Bivariate discrete copulas} \label{sec:bivdiscrcop}

Consider two ordered, finite sets $\Xs_1 = \{x^{(1)}_1,x^{(1)}_2,\ldots,x^{(1)}_{r_1}\}$ and $\Xs_2=\{x^{(2)}_1,x_2^{(2)},\ldots,x^{(2)}_{r_2}\}$, with $2 \leq r_1,r_2 < \infty$ and $x^{(1)}_1 < x^{(1)}_2 < \ldots < x^{(1)}_{r_1}$, $x^{(2)}_1 < x_2^{(2)} < \ldots < x^{(2)}_{r_2}$. The distribution of any bivariate discrete vector $(X_1,X_2)$ taking values in  $\Xs=\Xs_1 \times \Xs_2$ is a probability measure $P$ characterised by its probability mass function (pmf); viz.\ the function $\mathsf{p}: \Xs \to [0,1]:\mathsf{p}(x_{i_1}^{(1)},x_{i_2}^{(2)}) = P(\{(x_{i_1}^{(1)},x_{i_2}^{(2)})\}) = \P(X_1=x_{i_1}^{(1)},X_2=x_{i_2}^{(2)})$ ($(i_1,i_2) \in \{1,\ldots,r_1\} \times \{1,\ldots,r_2\}$). Such pmf can be broken down into three elements: $(i)$ what will be called the {\it probability array}
\begin{equation} \pp = \begin{pmatrix}
		p_{11} & p_{12} &   \ldots & p_{1 r_2} \\
		p_{21} & p_{22} &   \ldots & p_{2 r_2} \\
		\vdots & \vdots &  \ddots  & \vdots \\
		p_{r_1 1} & p_{r_1 2} &   \ldots & p_{r_1 r_2} 
	\end{pmatrix} \in [0,1]^{r_1 \times r_2},  \label{eqn:ppRS} \end{equation}
where $p_{i_1 i_2}  \doteq \mathsf{p}(x_{i_1}^{(1)},x_{i_2}^{(2)})$; $(ii)$ a {\it rank function} $\Phi_1: \Xs_1 \to \{1,\ldots,r_1\}$ and $(iii)$ a rank function $\Phi_2: \Xs_2 \to \{1,\ldots,r_2\}$ such that $\Phi_\ell(x^{(\ell)}_{i_\ell}) = i_\ell$, for $\ell = 1,2$ and $i_\ell \in \{1,\ldots,r_\ell\}$. We have $\mathsf{p}(x_{i_1}^{(1)},x_{i_2}^{(2)}) =  p_{\Phi_1(x^{(1)}_{i_1})\Phi_2(x^{(2)}_{i_2})}$, or equivalently $p_{i_1 i_2} = \mathsf{p}\left(\Phi_1^{-1}(i_1),  \Phi_2^{-1}(i_2) \right)$, for all $(i_1,i_2)$. 

\ppn Unlike the pmf $\mathsf{p}$, the probability array $\pp$ does not involve the initial domains $\Xs_1$ and $\Xs_2$, which are only specified by $\Phi_1, \Phi_2$. In fact, $\pp$ is what have in common the distributions of all vectors $(X_1',X_2')=(\Psi_1(X_1),\Psi_2(X_2))$ where $\Psi_1:\Xs_1 \to \Xs'_1$ and $\Psi_2: \Xs_2 \to \Xs_2'$ are any two increasing functions, and $\Xs'_1$ and $\Xs'_2$ are arbitrary ordered sets of cardinalities $r_1$ and $r_2$, respectively, as the ranks returned by $\Phi_\ell$ ($\ell =1,2$) are maximally invariant under the above group of increasing transformations \citep[Example 6.2.2(ii)]{Lehmann05}. Thus, if a representation or quantification of the dependence in the vector $(X_1,X_2)$ is to be invariant under such increasing marginal transformations -- which is precisely what is called `margin-free' in the copula literature -- then it must be based on $\pp$ and on $\pp$ only. 

\ppn Let $\bm p^{(1)} = (p_{1\bullet},p_{2\bullet},\ldots,p_{r_1 \bullet})$ and $\bm p^{(2)} = (p_{\bullet 1},p_{ \bullet 2},\ldots,p_{ \bullet r_2})$ be the {\it margins} of $\pp$: $p_{i_1 \bullet} = \sum_{i_2=1}^{r_2} p_{i_1 i_2} = \P(X_1 = x_{i_1}^{(1)})$ and $p_{\bullet i_2}= \sum_{i_1=1}^{r_1} p_{i_1 i_2} = \P(X_2 = x_{i_2}^{(2)})$. We assume throughout that  $p_{i_1 \bullet} p_{\bullet i_2} >0 \ \forall (i_1,i_2)$; i.e., no row or column of $\pp$ is identically null (otherwise these could simply be deleted). We define the {\it support} of $\pp$ as $\Ss_\pp = \Supp(\pp) =  \{(i_1,i_2)\in \{1,\ldots,r_1\}\times \{1,\ldots,r_2\}: p_{i_1 i_2} > 0 \}$ and let $N(\Ss_\pp) = \{\nu_1\times\nu_2\subset \{1,\ldots,r_1\}\times \{1,\ldots,r_2\}: (\nu_1\times\nu_2)\cap\Ss_\pp=\varnothing\}$, the {\it rectangular} subsets of null entries of $\pp$.

\ppn Now, let $\Gamma$ be the set of $(r_1 \times r_2)$-probability arrays (\ref{eqn:ppRS}) whose margins are uniform; viz.
\begin{equation} \Gamma \doteq \left\{\ggamma \in [0,1]^{r_1 \times r_2}:  \forall i_1, i_2,\  \gamma_{i_1 i_2} \geq 0, \gamma_{i_1\bullet} = \frac{1}{r_1}, \gamma_{\bullet i_2} = \frac{1}{r_2} \right\}. \label{eqn:gammadef} \end{equation} 
Note that $\Gamma$ is a convex subset of a vector space of dimension $(r_1-1)(r_2-1)$, known as a generalised {\it Birkhoff polytope} \citep{Brualdi06,Perrone2019}. For a given $\pp$, denote $\Gamma(\Ss_\pp) = \{\ggamma \in \Gamma: \Supp(\ggamma) = \Ss_\pp \} \subset \Gamma$, and $|A|$ the cardinality of any finite set $A$. 
\begin{lemma}\label{lem:Gamma-nonempty}
	For $\nu_\ell \subset \{1,\ldots,r_\ell\}$, write $\nu_\ell^c=\{1,\ldots,r_\ell\}\setminus\nu_\ell$  ($\ell=1,2$). The set $\Gamma(\Ss_\pp)$ is non-empty if and only if every $\nu_1\times\nu_2\in N(\Ss_\pp)$ satisfies
	\begin{equation}\label{eqn:condbiv}
		\frac{|\nu_1|}{r_1}+\frac{|\nu_2|}{r_2}\leq 1,\qquad\text{with equality only if }\nu_1^c\times\nu_2^c\in N(\Ss_\pp).
	\end{equation}
\end{lemma}
\noindent Under this form, the result appeared in \citet[Theorem 6.1]{Geenens2020}, but its substance originates from \cite{Ryser60} and \cite{Bacharach65}; see also \citet[Theorem 8.1.7]{Brualdi06}. Condition (\ref{eqn:condbiv}) means that the null entries in \eqref{eqn:ppRS} do not form a `bulky' set: there is no rectangular block of zeros sufficiently large to violate the inequality. Under condition (\ref{eqn:condbiv}), $\Gamma(\Ss_\pp) \neq \emptyset$, and it appears from \citet[Theorem 6.2]{Geenens2020} that there exists a {\it unique} probability array $\ggamma_\pp \in \Gamma(\Ss_\pp) $ such that 
\begin{equation} \ggamma_\pp  \doteq \arg \min_{\ggamma \in \Gamma} \sum_{(i_1,i_2) \in \Supp(\pp)} \gamma_{i_1 i_2} \log \frac{\gamma_{i_1 i_2}}{p_{i_1 i_2}} = \arg \min_{\ggamma \in \Gamma} \Is(\ggamma||\pp), \label{eqn:defcopKL} \end{equation}
where $\Is(\cdot\|\cdot)$ will be called `$\Is$-divergence' here following \cite{Csiszar75}, but also goes by `Kullback-Leibler divergence' or `Mutual Entropy'. Thus, $\ggamma_\pp \in \Gamma(\Ss_\pp)$ is the element of $\Gamma$ the closest to $\pp$ in the $\Is$-divergence sense. Furthermore, results of \citet[Section 3]{Csiszar75} and \cite{Ruschendorf93} on `$\Is$-projections' such as (\ref{eqn:defcopKL}) characterise $\ggamma_\pp$ as the unique probability array of $\Gamma$ satisfying
\begin{equation} \gamma_{i_1 i_2} = p_{i_1 i_2} \alpha_{i_1} \beta_{i_2}  \quad \forall (i_1, i_2) \label{eqn:alphabeta} \end{equation} 
for two sets of positive constants $\{\alpha_{i_1}\}_{i_1=1}^{r_1}$ and $\{\beta_{i_2}\}_{i_2=1}^{r_2}$. The probability measure $G_\pp$, reconstructed from the probability array $\ggamma_\pp$ with the inverse rank functions $\Phi^{-1}_\ell(i_\ell) = \frac{i_\ell}{r_\ell+1}$ ($i_\ell=1,\ldots,r_\ell$, $\ell = 1,2$), captures the distribution of a vector $(U_1,U_2)$ taking values in $\{\frac{1}{r_1+1},\ldots,\frac{r_1}{r_1+1} \} \times \{\frac{1}{r_2+1},\ldots,\frac{r_2}{r_2+1} \}$, such that $\P(U_1 = \frac{i_1}{r_1+1}) = \frac{1}{r_1}$ and $\P(U_2 = \frac{i_2}{r_2+1}) = \frac{1}{r_2}$ (discrete uniform marginal distributions) and, from (\ref{eqn:alphabeta}),
\begin{equation} \P\left(U_1=\frac{i_1}{r_1+1} , U_2 = \frac{i_2}{r_2+1}\right) = \P\left(X_1 = x^{(1)}_{i_1}, X_2= x^{(2)}_{i_2}\right)\alpha_{i_1} \beta_{i_2}. \label{eqn:alphabeta2} \end{equation}

\ppn In the case of a rectangular support for $(X_1,X_2)$ (i.e., all $p_{i_1 i_2} > 0$ in $\pp$), then (\ref{eqn:alphabeta})-(\ref{eqn:alphabeta2}) is equivalent to saying that $(X_1,X_2) \sim P$ and $(U_1,U_2) \sim G_\pp$ admit the same sets of $(r_1-1)(r_2-1)$ odds ratios
\begin{equation*} 
	\frac{\gamma_{11}\gamma_{i_1i_2}}{\gamma_{i_11} \gamma_{1i_2}} = \frac{p_{11} \alpha_1 \beta_1p_{i_1i_2} \alpha_{i_1} \beta_{i_2}}{p_{i_1 1}\alpha_{i_1} \beta_1 p_{1i_2}\alpha_1 \beta_{i_2}}= \frac{p_{11}p_{i_1i_2}}{p_{i_11} p_{1i_2}}  \quad \forall (i_1,i_2) \in \{2,\ldots,r_1\} \times \{2,\ldots,r_2\} , \label{eqn:ORs} \end{equation*}
which unequivocally describe the {\it dependence structure} in such bivariate discrete distributions (\citealp[Chapter 2]{Bishop75}; \citealp[Chapter 2]{Agresti13}; \citealp[Section 4.3.1(a)]{Geenens22}). Thus, $P$ and $G_\pp$ admit the same dependence structure.

\ppn More generally, the number and layout of null entries in $\pp$ form an essential element of the dependence structure of the vector $(X_1,X_2)$: the presence of any $p_{i_1i_2} = 0$ in (\ref{eqn:ppRS}) ({\it non-rectangular support})  automatically prevents independence between $X_1$ and $X_2$: algebraically, $0 = p_{i_1i_2}  \neq p_{i_1 \bullet} p_{\bullet i_2}$, as it has been assumed $p_{i_1 \bullet}, p_{\bullet i_2} >0$ for all $i_1,i_2$; intuitively, the fact that some values of $X_1$ and $X_2$ are not allowed simultaneously can be understood as a strong form of interaction between $X_1$ and $X_2$, indeed. 

\ppn This contribution to the dependence exclusively due to the shape of $\Ss_\pp$ -- what \cite{Holland86} called `{\it regional dependence}' -- is maintained in $\ggamma_\pp$ through (\ref{eqn:defcopKL}) as $\ggamma_\pp \in \Gamma(\Ss_\pp)$. The rest of the dependence structure (not of `regional' nature) remains unequivocally described by a certain number of odds-ratio-like quantities which essentially `avoid' the null entries $p_{i_1i_2} =0$; see Example \ref{ex:example}  below for a simple illustration, or \citet[Section 4.3.1(b)]{Geenens22} for full details. Again, these odds-ratios are maintained through (\ref{eqn:alphabeta}), and thus shared by $\pp$ and $\ggamma_\pp$. It follows that $P$ and $G_\pp$ admit the same dependence structure in the case of non-rectangular support as well.

\ppn In fact, (\ref{eqn:alphabeta2}) is what \citet[Theorem 4.1]{Geenens22} established as a necessary and sufficient condition for $(U_1,U_2) \sim G_\pp$ and $(X_1,X_2) \sim P$ to share the same dependence structure. Thus, $(U_1,U_2)$ is the random vector whose dependence structure is identical to that of $(X_1,X_2)$, but with (discrete) uniform marginals on $\{\frac{1}{r_1+1},\ldots,\frac{r_1}{r_1+1}\}$ and $\{\frac{1}{r_2+1},\ldots,\frac{r_2}{r_2+1}\}$ -- and this, irrespective of the initial marginal domains and distributions of $(X_1,X_2)$. In other words, $G_\pp$ is a margin-free distribution-morphic representation of that dependence structure, in a way obviously reminiscent of the (unique) continuous copula $C_{X_1X_2}$ of $(X_1,X_2)$ if it were a continuous bivariate vector. This justifies to call $G_\pp$ the {\it discrete copula} measure of $(X_1,X_2) \sim P$. By analogy with the continuous case, we will call `discrete copula' the joint cumulative distribution function of $(U_1,U_2) \sim G_\pp$ and `copula pmf' its probability mass function, as in \cite{Geenens2020} and \cite{Kojadinovic24}. Accordingly, we will call $\ggamma_\pp$ the {\it copula array} of $\pp$.

\ppn Importantly, under the assumption $\Gamma(\Ss_\pp) \neq \emptyset$, the transformation (\ref{eqn:defcopKL}) is reversible: if $\gamma_\pp$ is the $\Is$-projection of $\pp$ onto the class $\Gamma$ of probability arrays with uniform margins, then $\pp$ is the $\Is$-projection of $\gamma_\pp$ onto the corresponding \emph{Fr\'echet class} of all probability arrays sharing the same margins as $\pp$ -- Proposition \ref{prop:sklar:like} below formalises this duality in full generality for the $d$-dimensional case. Furthermore, $\gamma_\pp$ can be $\Is$-projected onto other Fr\'echet classes (probability arrays sharing other margins), which produces new probability arrays that retain the dependence structure of $\pp$ while altering the marginals \citep{Geenens2020,Kojadinovic24}. This is the very essence of classical copula modelling.

\begin{remark}
	The definition of a discrete copula given above is consistent with those found in standard references on the topic -- see, for instance, \citet[Definition~3.1.7]{Durante15}, \citet[Definition~2]{Kolesarova2006}, and \citet[Section~2.1]{Perrone2019} -- up to a normalisation convention. Specifically, while these references work with grids of the form $\{\tfrac{1}{r_1},\ldots,\tfrac{r_1-1}{r_1},1\}\times\{\tfrac{1}{r_2},\ldots,\tfrac{r_2-1}{r_2},1\}$, we use $\{\tfrac{1}{r_1+1},\ldots,\tfrac{r_1}{r_1+1}\}\times\{\tfrac{1}{r_2+1},\ldots,\tfrac{r_2}{r_2+1}\}$, a slight renormalisation that is classical in empirical copula constructions as it keeps support points away from the boundaries of $[0,1]^2$. Other contributions (e.g.\ \citealp{Perrone22}) define discrete copulas on arbitrary grids, often in connection with subcopulas; this departs from the core continuous copula definition of having uniform margins on $[0,1]$, a property we seek to mimic in order to retain a faithful discrete analogue. \qed
\end{remark}

\begin{remark} \label{rem:mincond} The discrete copula construction is described under the assumption $\Gamma(\Ss_\pp) \neq \emptyset$; that is, under condition (\ref{eqn:condbiv}). We can wonder what happens if that condition is not met. There are two cases to distinguish: $(a)$ $\Gamma(\Ss_\pp) = \emptyset$ but there exists $\Ss' \subset \Ss_\pp$ such that $\Gamma(\Ss') \neq \emptyset$; and $(b)$ $\Gamma(\Ss') = \emptyset$ for all $\Ss' \subseteq \Ss_\pp$. We discuss these two cases in Appendix \ref{app:examples} (Supplementary Material), through examples. \qed 
\end{remark}

\section{$d$-variate discrete copulas} \label{sec:dvarcop}

This section describes the extension of the above bivariate construction to any dimension $d \in \N$. 

\subsection{$d$-dimensional probability arrays}

For simplicity, below we denote $[s] \doteq \{1,\ldots,s\}$ for any $s \in \N$. For given $d \in \N$ and $2 \leq r_1, \dots, r_d < \infty$, let $\R^{r_1 \times \cdots \times r_d}$ be the set of all $d$-dimensional arrays whose dimension sizes are $r_1, \dots, r_d$. 
Any array $\ab \in \R^{r_1 \times \cdots \times r_d}$ can be expressed explicitly `entrywise' in terms of its elements as $(a_{\bm i})_{\bm i \in \Rs}$, where 
\begin{equation} \Rs = [r_1] \times \dots \times [r_d] = \left\{ {\bm i} = (i_1, i_2, \dots i_d) : i_1 \in [r_1], \dots, i_d \in [r_d] \right\}. \label{eq:Rs} \end{equation}
The {\it support} of such array is then defined as $\Supp(\ab) = \left\{ \bm i \in \Rs : a_{\bm i} \neq 0 \right\}$. When summing the elements of $\ab$ over all dimensions but one, we obtain a one-dimensional array (that is, a vector). Specifically, for any $\ell \in [d]$, we set 
\begin{equation} a^{(\ell)}_i = \sum_{ i_1 \in [r_1]} \dots \sum_{ i_{\ell-1} \in [r_{\ell-1}]} \sum_{ i_{\ell+1} \in [r_{\ell+1}]} \dots \sum_{ i_d \in [r_d]} a_{i_1, \dots, i_{\ell-1}, i,  i_{\ell+1}, \dots ,i_d}, \qquad i \in [r_{\ell}], \label{eqn:amarg} \end{equation}
corresponding to the $r_\ell$ components of the marginal sum vector when summing the elements of $\ab$ over all the dimensions but the $\ell$th one. We will write $\bm a^{(\ell)} = (a^{(\ell)}_1,\ldots,a^{(\ell)}_{r_\ell})$ and call it the $\ell^\text{th}$ margin of $\ab$. 

\ppn For any $\ell \in [d]$, let $\Xs_\ell = \{x^{(\ell)}_1,\dots, x^{(\ell)}_{r_\ell}\}$ with $x^{(\ell)}_1 < \dots < x^{(\ell)}_{r_\ell}$. To any probability measure $P$ defined on $\Xs = \Xs_1 \times  \cdots \times \Xs_d$, we may associate $(i)$ $d$ rank functions $\Phi_\ell: \Xs_\ell \to \{1,\ldots,r_\ell\}$ ($\ell \in [d]$) such that $\Phi_\ell(x^{(\ell)}_{i_\ell}) = i_\ell$ ($i_\ell \in [r_\ell])$, and $(ii)$ a {probability array}; i.e., the element $\pp$ of $[0,1]^{r_1 \times \cdots \times r_d}$ defined by
$$
p_{\bm i} = P \left( \left\{ (x_{i_1}^{(1)}, \dots, x_{i_d}^{(d)}) \right\} \right), \qquad \bm i = (i_1, i_2, \dots i_d) \in \Rs.
$$
Like in Section \ref{sec:bivdiscrcop}, the probability array $\pp$ represents the constituent of the distribution of a random vector $\XX=(X_1,\ldots,X_d) \sim P$ which remains invariant under increasing transformations of its individual components. Consequently, any attempt at describing the dependence structure of $(X_1,\ldots,X_d)$ in a `margin-free' way should be based on $\pp$ and on $\pp$ only. 

\ppn Probability arrays are necessarily such that $ p_{\bm i} \geq 0$ ($\forall \bm i \in \Rs$) and $\sum_{\bm i \in \Rs} p_{\bm i} = 1$.   
We define the $d$ margins of a probability array $\pp$ to be the $d$ univariate probability arrays (i.e., vectors) $\bm p^{(1)}, \dots, \bm p^{(d)}$, where ${\bm p}^{(\ell)} = (p^{(\ell)}_1, \dots, p^{(\ell)}_{r_\ell} ) \in \R^{r_\ell}$ ($\ell \in [d]$) and $p^{(\ell)}_i$ ($i \in [r_{\ell}]$) is defined as in (\ref{eqn:amarg}). Without loss of generality, it will be assumed throughout that any probability array $\pp$ under consideration belongs to the set
\begin{equation} \Ps = \left\{ \pp \in [0,1]^{r_1 \times \cdots \times r_d} :   \, \sum_{\bm i \in \Rs} p_{\bm i} = 1 \text{ and } p^{(\ell)}_i> 0  \text{ for all } i \in [r_{\ell}] \text{ and } \ell \in [d] \right\}, \label{eq:Ps} \end{equation}
that is, probability arrays which do not include any identically null slice. If it happens that $p^{(\ell)}_i = 0$ for some $\ell \in [d],i \in [r_{\ell}]$, it suffices to remove the element $x^{(\ell)}_i$ from $\Xs_\ell$ and decrement $r_\ell$ by 1.

\subsection{Fréchet classes and $\Is$-projections} \label{subsec:IprojFrech}

Consider $d$ univariate positive probability arrays (that is, vectors whose entries are positive and sum to one) $\bm q^{(1)}, \dots, \bm q^{(d)}$ in $(0,1)^{r_1},\dots,(0,1)^{r_d}$, respectively, and call 
\begin{equation*} \Fs(\bm q^{(1)}, \dots, \bm q^{(d)}) = \{ \pp \in \Ps : \bm p^{(\ell)} = \bm q^{(\ell)}, \  \forall\,\ell \in [d] \} \label{eq:frechet} \end{equation*}
the {\it Fr\'echet class} of probability arrays in $[0,1]^{r_1 \times \cdots \times r_d}$ with margins $\bm q^{(1)}, \dots, \bm q^{(d)}$. Any such class $\Fs(\bm q^{(1)}, \dots, \bm q^{(d)})$ is necessarily a non-empty subset of $\Ps$ in \eqref{eq:Ps}, as it contains at least the `independence' array denoted $[\bm q^{(1)} \bigcdot \dots \bigcdot \bm q^{(d)}]$ and defined as
\begin{equation} [\bm q^{(1)} \bigcdot \dots \bigcdot \bm q^{(d)}]_{\bm i} = \prod_{\ell = 1}^d q_{i_\ell}^{(\ell)}, \quad \bm i = (i_1, i_2, \dots i_d) \in \Rs. \label{eqn:indeparray} \end{equation} 
In fact, we have the following standard but essential result:
\begin{lemma} \label{lem:frechet} For any $\bm q^{(1)} \in (0,1)^{r_1},\dots,\bm q^{(d)} \in (0,1)^{r_d}$ univariate positive probability arrays, the Fr\'echet class $\Fs(\bm q^{(1)},\dots,\bm q^{(d)}) \subset \Ps$ is non-empty, convex, and closed in total variation.
\end{lemma}

\begin{proof}  See Appendix \ref{app:proofs} in Supplementary Material. \end{proof}

\ppn For any two probability arrays $\pp, \qq \in \Ps$, the $\Is$-divergence of $\qq$ with respect to $\pp$ is defined as
\begin{equation} \Is(\qq \| \pp) = \begin{cases}
		\displaystyle \sum_{\bm i \in \Supp(\pp)} q_{\bm i} \log{\frac{q_{\bm i}}{p_{\bm i}}}  &\text{ if }\Supp{(\qq)} \subseteq \Supp{(\pp)},\\
		\infty&\text{ otherwise,}
	\end{cases} \label{eq:KL:divergence} \end{equation}
with the convention that $0 \log 0 = 0$. We say that $\qq^* \in \Fs(\bm q^{(1)}, \dots, \bm q^{(d)})$ is an $\Is$-projection of $\pp$ on $\Fs(\bm q^{(1)}, \dots, \bm q^{(d)})$ if
\begin{equation*}
	\Is(\qq^* \| \pp) = \inf_{\qq \in \Fs(\bm q^{(1)}, \dots, \bm q^{(d)})} \Is(\qq \| \pp).
\end{equation*}
Clearly, $\Is(\qq \| \pp) < \infty \iff \Supp{(\qq)} \subseteq \Supp{(\pp)}$. Thus, if $\Fs(\bm q^{(1)}, \dots, \bm q^{(d)})$ does not contain any array $\qq$ such that $\Supp{(\qq)} \subseteq \Supp{(\pp)}$, then all $\qq \in \Fs(\bm q^{(1)}, \dots, \bm q^{(d)})$ are such that $\Is(\qq \| \pp) = \infty$ and $\pp$ does not admit a (proper) $\Is$-projection. By contrast, if $\Fs(\bm q^{(1)}, \dots, \bm q^{(d)})$ contains at least one probability array $\qq$ such that $\Supp(\qq) \subseteq \Supp(\pp)$, then any $\Is$-projection $\qq^*$ is necessarily such that $\Supp(\qq^*) \subseteq \Supp(\pp)$. In fact, the strict convexity of $\Is$ as a function of $\qq$ and the convexity and closedness of $\Fs(\bm q^{(1)}, \dots, \bm q^{(d)})$ guarantee the uniqueness of that $\Is$-projection \citep[Theorem 2.1]{Csiszar75}, which we will then denote
\begin{equation} \qq^* \doteq \Is_{\Fs(\bm q^{(1)}, \dots, \bm q^{(d)})}(\pp).  	\label{eq:I:proj:Frechet} \end{equation}
Under a slightly stronger assumption, \citet[Theorem 3.1 and Corollary 3.1]{Csiszar75} provides a simple characterisation of this unique $\Is$-projection, recalled here for convenience:

\begin{lemma} \label{lem:multrepres} Let $\pp \in \Ps$, and assume that there exists $\qq \in \Fs(\bm q^{(1)}, \dots, \bm q^{(d)})$ such that $\Supp(\qq) = \Supp(\pp)$. Then $\qq^* = \Is_{\Fs(\bm q^{(1)}, \dots, \bm q^{(d)})}(\pp)$ if and only if $\qq^* \in  \Fs(\bm q^{(1)}, \dots, \bm q^{(d)})$ and there exists sets of positive constants $\{\alpha^{(\ell)}_1,\ldots,\alpha^{(\ell)}_{r_\ell}\}$ ($\ell = 1,\ldots,d$) such that 
	\begin{equation} q^*_{\bm i} = p_{\bm i} \prod_{\ell =1}^d \alpha^{(\ell)}_{i_\ell}, \qquad \forall \bm i = (i_1,\ldots,i_d)\in \Rs. \label{eqn:alphasd} \end{equation}
\end{lemma}
\noindent Since the constants $\{\alpha^{(\ell)}_{i_\ell}\}$ in (\ref{eqn:alphasd}) are all positive, it is immediate to see:
\begin{corollary} \label{cor:suppproj}  Let $\pp \in \Ps$, and assume that there exists $\qq \in \Fs(\bm q^{(1)}, \dots, \bm q^{(d)})$ such that $\Supp(\qq) = \Supp(\pp)$. Then the $\Is$-projection $\qq^* = \Is_{\Fs(\bm q^{(1)}, \dots, \bm q^{(d)})}(\pp)$ is such that $\Supp(\qq^*) = \Supp(\pp)$.
\end{corollary}

\begin{remark} \label{rmk:IPF} Obtaining the $\Is$-projection of a probability array onto a Fréchet class in closed form is generally impossible, with explicit solutions available only in simple cases (Example \ref{ex:example} below provides one such instance). However, when an $\Is$-projection exists, it can always be numerically approximated using the Iterated Proportional Fitting (IPF) procedure  \citep{Ireland68,Fienberg70,Ruschendorf95,Idel16}. This algorithm, also known as Sinkhorn algorithm \citep{Sinkhorn64,Sinkhorn67}, offers fast convergence under the conditions of Lemma \ref{lem:multrepres}; see \citet[Theorem 3.2]{Brossard18}. Practical implementation (for any dimension $d$) is provided by the R package {\tt mipfp} \citep{Barthelemy18}; see Appendix \ref{app:Rscript} for a sample of code. Consequently, we can consider any required $\Is$-projection to be effectively accessible to us. \qed 
\end{remark}

\subsection{Copula arrays and discrete copulas}

For all $\ell \in [d]$, let $\bm u^{(\ell)} \doteq (\frac{1}{r_\ell},\ldots,\frac{1}{r_\ell})$, the univariate probability arrays corresponding to (discrete) uniform distributions. We define the Fr\'echet class  
\begin{equation} \Gamma \doteq \Fs({\bm u}^{(1)},\ldots,{\bm u}^{(d)}) = \left\{ \ggamma \in \Ps : \gamma^{(\ell)}_i = \frac{1}{r_\ell} \text{ for all } i \in [r_\ell] \text{ and } \ell \in [d] \right\}, 	\label{eq:Gamma} \end{equation}
and call its elements \emph{copula arrays}. We note that $\Gamma$ is related to the convex polytope of the polystochastic arrays of $d$-dimensions (`$d$-dimensional Birkhoff polytope'; \citealp{Jurkat68,Linial14,Taranenko16}). 

\ppn For any non-empty subset $\Ss \subseteq \Rs$ in (\ref{eq:Rs}), let
\begin{align} \Ps(\Ss)  = \left\{ \pp \in \Ps : \Supp(\pp) = \Ss \right\} \qquad \text{ and } \qquad  \Gamma(\Ss)  = \left\{ \ggamma \in \Gamma : \Supp(\ggamma) = \Ss \right\}. \label{eq:Gamma:S}
\end{align}
Consider a given $\pp \in \Fs({\bm p^{(1)},\dots,\bm p^{(d)}}) \subset \Ps$, whose support is $\Ss_{\pp} = \Supp(\pp)$. According to Section \ref{subsec:IprojFrech}, $\Gamma(\Ss_{\pp}) \neq \emptyset$ is a sufficient condition for $\pp$ to admit a unique $\Is$-projection on $\Gamma$; viz. 
\begin{equation} \ggamma_{\pp} \doteq \Is_{\Gamma}(\pp), \label{eqn:copproj} \end{equation}
where $\Is_{\Gamma}(\cdot) = \Is_{\Fs({\bm u}^{(1)},\ldots,{\bm u}^{(d)})}(\cdot)$ is defined as in~\eqref{eq:I:proj:Frechet}. Explicitly, from Lemma \ref{lem:multrepres} and Corollary \ref{cor:suppproj}, $\ggamma_\pp \in \Gamma(\Ss_\pp)$ and is the unique element of $\Gamma$ such that 
\begin{equation} \gamma_{\pp;\bm i} = p_{\bm i} \prod_{\ell =1}^d \beta_{i_\ell}^{(\ell)} \label{eqn:copprojbetas} \end{equation}
for some sets of positive constants $\{\beta_{i_\ell}^{(\ell)}\}$, $\ell \in [d], i_\ell \in [r_\ell]$. Reversing (\ref{eqn:copprojbetas}), we have directly 
\begin{equation} p_{\bm i} = \frac{\gamma_{\pp;\bm i}}{\prod_{\ell =1}^d \beta_{i_\ell}^{(\ell)}} = \gamma_{\pp;\bm i}\prod_{\ell =1}^d \frac{1}{\beta_{i_\ell}^{(\ell)}}, \label{eqn:copprojinvbetas} \end{equation}
and as $\{1/\beta_{i_\ell}^{(\ell)}\}$ ($\ell \in [d], i_\ell \in [r_\ell]$) are positive constants as well, we deduce from Lemma \ref{lem:multrepres} that $\pp$ is the unique $\Is$-projection of $\gamma_\pp$ on $\Fs({\bm p^{(1)},\dots,\bm p^{(d)}})$. The following result -- essentially the $d$-variate analog to \citet[Proposition~3.1]{Kojadinovic24} -- formalises the previous observations.

\begin{proposition} \label{prop:sklar:like}
	Let $\pp \in \Fs(\bm p^{(1)},\ldots,\bm p^{(d)}) \subset \Ps$ have support $\Ss_{\pp} \subseteq \Rs$. Then the following assertions are equivalent:
	\begin{itemize}
		\item[(i)] $\Gamma(\Ss_{\pp}) \neq \emptyset$;
		\item[(ii)] $\pp$ admits a unique $\Is$-projection on $\Gamma$, viz.\  $\ggamma_\pp = \Is_\Gamma(\pp)$, and $\Supp(\ggamma_\pp) = \Ss_{\pp}$; 
		\item[(iii)] there exists a unique $\ggamma\in\Gamma$ such that $\pp=\Is_{\Fs(\bm p^{(1)},\ldots,\bm p^{(d)})}(\ggamma)$, and this $\ggamma$ necessarily coincides with $\ggamma_\pp=\Is_\Gamma(\pp)$.
	\end{itemize}
\end{proposition}
\begin{proof}  See Appendix \ref{app:proofs} in Supplementary Material. \end{proof}

\noindent Any probability array $\pp$ (such that $ \Gamma(\Ss_\pp) \neq \emptyset$) may thus be understood as formed by a copula array `seed' $\ggamma_\pp \in \Gamma$ on which relevant margins  ${\bm p^{(1)},\dots,\bm p^{(d)}}$ have been `mounted' via (\ref{eqn:copprojinvbetas}). Reversely, that seed $\ggamma_\pp$ may be recovered by `unmounting' those margins as in (\ref{eqn:copprojbetas}), which is equivalent to `mounting' (via the same process) uniform margins on $\pp$. From a statistical modelling perspective, the reversibility of the construction  
\begin{equation} \pp = \Is_{\Fs({\bm p^{(1)},\dots,\bm p^{(d)}})}(\ggamma_\pp) \iff \gamma_\pp = \Is_{\Gamma}(\pp) \label{eqn:sklarlike} \end{equation}  
is crucial, as it enables the development of flexible models by initially concentrating on the dependence/copula and subsequently transitioning back to the original marginal scales (`{\it there and back again}'). This is naturally reminiscent of the classical copula methodology articulated around Sklar's representation in continuous cases. In fact, (\ref{eqn:sklarlike}) may be interpreted -- to some extent -- as a discrete analogue to the usual Sklar's construction, with $\Is$-projection on Fr\'echet classes playing the role of marginal transformations.

\ppn Now, for any margins $\bm q^{(1)}, \dots, \bm q^{(d)}$ such that $\exists \qq \in \Fs(\bm q^{(1)},\ldots,\bm q^{(d)})$ with $\Supp(\qq) = \Ss_{\pp}$, the unique $\Is$-projection $\qq^*$ of $\pp$ on $\Fs(\bm q^{(1)},\ldots,\bm q^{(d)})$ satisfies (\ref{eqn:alphasd}), and we conclude with (\ref{eqn:copprojinvbetas}):
\begin{equation} q^*_{\bm i} = p_{\bm i} \prod_{\ell =1}^d \alpha^{(\ell)}_{i_\ell} = \gamma_{\pp;\bm i} \prod_{\ell =1}^d \frac{\alpha^{(\ell)}_{i_\ell} }{\beta_{i_\ell}^{(\ell)}}, \qquad \forall \bm i \in \Rs. \label{eqn:groupaction} \end{equation}
As $\alpha^{(\ell)}_{i_\ell}/\beta_{i_\ell}^{(\ell)} > 0$ for all $\ell \in[d]$ and $i_\ell \in [r_\ell]$, this means that $\qq^*$ is the $\Is$-projection of $\ggamma_\pp$ on $\Fs(\bm q^{(1)},\ldots,\bm q^{(d)})$ as well (Lemma \ref{lem:multrepres}), and by Proposition \ref{prop:sklar:like}, $\ggamma_\pp$ is the copula array of $\qq^*$. In fact, as (\ref{eqn:alphasd}) defines a group action on the space of probability arrays $\Ps$, it defines an equivalence relation:
\[\qq \sim \pp  \iff \exists \{\alpha^{(\ell)}_{i_\ell}>0\}_{\ell \in [d], i_\ell \in [r_\ell]} \ \text{s.t. }\  q_{\bm i} = p_{\bm i} \prod_{\ell =1}^d \alpha^{(\ell)}_{i_\ell} \ \forall \bm i \in \Rs.   \]
In cases where $\Ss \doteq \Ss_\qq = \Ss_\pp$ is such that $\Gamma(\Ss) \neq \emptyset$, then $\qq \sim \pp \iff \ggamma_\qq = \ggamma_\pp$. The copula array is thus a maximal invariant under the above group operation, which captures `everything but the margins' of the probability arrays under consideration -- that is, according to \cite{Geenens22}, the underlying dependence structure. For example, if $\Supp(\pp) = \Rs$ ({\it hyper-rectangular support}), then it follows from (\ref{eqn:groupaction}) that $\pp$, $\qq^*$ and $\ggamma_\pp = \ggamma_{\qq^*}$ share the same sets of conditional and higher-order odds ratios as described in \citet[Sections 6.2 and 10.1]{Rudas18}.

\ppn Like in Section \ref{sec:bivdiscrcop}, we can now define the (discrete) copula measure $G_\pp$ of a vector $\XX=(X_1,\ldots,X_d) \sim P$ with probability array $\pp$ (such that $\Gamma(\Ss_\pp)\neq \emptyset$) as the probability measure of the vector $\UU=(U_1,\ldots,U_d)$ supported on $\bigtimes_{\ell \in [d]} \frac{[r_\ell]}{r_\ell+1} = \bigtimes_{\ell \in [d]} \{\frac{1}{r_\ell + 1}, \ldots, \frac{r_\ell}{r_\ell + 1} \}$ whose probability array is $\ggamma_\pp$; in other words, the discrete measure obtained by combining the probability array $\ggamma_\pp$ with the $d$ inverse rank functions $\Phi^{-1}_{\ell}: [r_\ell] \to \frac{[r_\ell]}{r_\ell+1}: \Phi^{-1}_\ell(i) = \frac{i}{r_\ell+1}$, $\ell \in [d]$. Such copula measure is a margin-free distribution-morphic representation of the dependence structure of the random vector $\XX=(X_1,\ldots,X_d) \sim P$. This justifies:

\begin{definition} \label{def:cop} Let $\XX=(X_1,\ldots,X_d) \sim P$ be a $d$-dimensional discrete random vector whose probability array $\pp \in \Ps$ has support $\Ss_\pp \subseteq \Rs$. If $\Gamma(\Ss_\pp) \neq \emptyset$, we call the unique copula array $\ggamma \in \Gamma$ satisfying $\ggamma = \Is_\Gamma(\pp)$ the copula array of $\pp$, denoted $\gamma_\pp$. The probability measure $G_\pp$ on  $\bigtimes_{\ell \in [d]} \frac{[r_\ell]}{r_\ell+1}$ whose probability array is $\ggamma_\pp$ is called the (discrete) copula measure of $\XX$, its joint cumulative distribution is the (discrete) copula of $\XX$ and the associated joint probability mass function is the copula pmf of $\XX$.
\end{definition}

\begin{remark}
	The copula array $\ggamma_\pp$ thus represents the relevant odds ratios of $\pp$ and, in practice, is computed via IPF (Remark~\ref{rmk:IPF}). This may be reminiscent of log-linear models in classical contingency-table analysis, which we can identify with multivariate discrete distributions. Log-linear models break down the joint distribution into a cascade of interaction terms, all expressed as (some form of) log-odds ratios, and are typically fitted using IPF as well \citep[Section~3.5]{Bishop75}.
	
	\ppn However, in the log-linear parameterisation, so-called main/marginal effects are \textit{not} parameters of the marginal distributions; they provide {\it contrasts} within the full joint distribution and do {not} identify the margins. Thus, despite the usual terminology (`marginal effects' versus `interactions'), log-linear models do not achieve a clean separation between marginal behaviour and dependence: all aspects of the joint distribution remain intertwined through the exponential-family structure. By contrast, the proposed discrete copula framework aims precisely to achieve, and take advantage of, such a separation. \qed
\end{remark}

\subsection{Parameterisation of copula arrays with a given support $\Ss$}

As for $d=2$ (Section \ref{sec:bivdiscrcop}), it appears through Proposition \ref{prop:sklar:like} that the condition $\Gamma(\Ss_\pp) \neq \emptyset$ ensures that the discrete copula machinery fits neatly into place for distributions with probability array $\pp$. 
The support $\Ss_\pp$ of $\pp$ is pivotal when identifying its copula array $\ggamma_\pp$, as $\ggamma_\pp$ admits the same support as $\pp$ (Proposition \ref{prop:sklar:like}$(ii)$). Consider a given non-empty subset $\Ss \subseteq \Rs$, and assume that:
\begin{condition} The set $\Gamma(\Ss)$ in~\eqref{eq:Gamma:S} is nonempty.   \label{cond:Gamma:S} \end{condition}

\begin{remark} \label{rmk:checkcond} If the bivariate case $d=2$ is well understood (Lemma~\ref{lem:Gamma-nonempty}), no comparably simple explicit combinatorial characterisation of Condition~\ref{cond:Gamma:S} is currently available when $d\ge3$. Exact necessary and sufficient criteria do exist  \citep{Franklin89,Friedland11}, but they are formulated as a non-trivial feasibility problem for the support $\Ss$, rather than through a simple closed-form combinatorial criterion. The trivariate Bernoulli setting ($d=3$; $r_1 = r_2 = r_3 = 2$) was considered in \citet{Perrone24}, while \citet{Fontana26} subsequently characterised admissible zero patterns for multiway Bernoulli tables through the extreme points of the generalised Birkhoff polytope and gave a linear-programming test for exact-support feasibility. In Appendix~\ref{app:checkCond31}, we formulate this linear-programming test for arbitrary marginal cardinalities ($r_\ell \geq 2, \ell \in [d], d \in \N$), and use it as a practical computational check of Condition~\ref{cond:Gamma:S}. \qed
\end{remark}

\noindent Under Condition \ref{cond:Gamma:S}, there exist copula arrays with support $\Ss$. In this section, we construct a general parameterisation of those elements of $\Gamma(\Ss)$. This parameterisation will be instrumental to define and study estimators of copula arrays in Section \ref{sec:empcop}.

\subsubsection{$\Gamma(\Ss)$ as a subset of an affine space} \label{sec:aff}

Take the probability array $\qq \in \Ps(\Ss)$ defined entrywise as 
\begin{equation} q_{\ii} = \frac{1}{|\Ss|}\indic{\ii \in \Ss}, \quad \bm i \in \Rs. \label{eqn:qquasi}\end{equation}
(Here $\indic{\bullet}$ denotes the indicator function, taking the value 1 if statement $\bullet$ is true and 0 otherwise.) Since $\Supp(\qq) = \Ss$, under Condition~\ref{cond:Gamma:S}, $\qq$ admits a unique $\Is$-projection on $\Gamma$ (Proposition \ref{prop:sklar:like}):
\begin{equation}  \ggamma^{(q,\Ss)} \doteq \Is_\Gamma(\qq) .  \label{eq:gamma:q}
\end{equation}
It follows from the definition of $\qq$ and (\ref{eqn:copprojbetas}) that 
\[\gamma_{\ii}^{(q,\Ss)}  = \left\{\prod_{\ell =1}^d \beta^{(\ell)}_{i_\ell}\right\} \indic{\bm i \in \Ss} \]
for some positive constants $\{\beta_{i_\ell}^{(\ell)}\}$, $\ell \in [d], i_\ell \in [r_\ell]$. If $\Ss = \Rs$ (hyper-rectangular support), then $\gamma_{\ii}^{(q,\Rs)}$ factorises as in (\ref{eqn:indeparray}) for all $\ii$, and $\ggamma^{(q,\Rs)}$ is simply the independence copula array $\ggamma^{(\indep)} \doteq [{\bm u}^{(1)} \bigcdot \dots \bigcdot {\bm u}^{(d)}]$ defined as
\[ \gamma_\ii^{(\indep)}  = \frac{1}{\prod_{\ell =1}^d  r_\ell}, \quad \forall \ii \in \Rs. \] 
In contrast, if $\Ss \neq \Rs$ (non hyper-rectangular support), then the dependence structure described by $\ggamma^{(q,\Ss)}$ is that of non-independent variables -- due to $\gamma_{\ii}^{(q,\Ss)} = 0$ for some $\ii \in \Rs$; recall Section \ref{sec:bivdiscrcop} -- which behave {\it as if they were} on their joint support (their joint distribution factorises on $\Ss$). This is what \cite{Goodman68} called `{quasi-independence}'; more on this in Section \ref{subsec:quasitest} below. Therefore, we will refer to $\ggamma^{(q,\Ss)}$ as the {\it quasi-independence copula array} with support $\Ss$. 

\ppn Let $\ggamma \in \Gamma(\Ss)$ and define
\begin{equation} \Delta \doteq \ggamma - \ggamma^{(q,\Ss)} \quad \in \R^{r_1 \times \cdots \times r_d}. 	\label{eq:gamma:rep} \end{equation}
As both $\ggamma$ and $\ggamma^{(q,\Ss)}$ are copula arrays of support $\Ss$, $\Delta$ actually belongs to the set
\begin{equation} \bm \Delta^\circ(\Ss) = \{ \ab \in \R^{r_1 \times \cdots \times r_d} : \Supp(\ab) \subseteq \Ss \text{ and } a^{(\ell)}_i= 0 \ \forall\,  \ell \in [d], \, i \in [r_\ell]   \} \label{eq:DDelta}
\end{equation}
of the arrays whose support is included in $\Ss$ and all univariate margins are null. As a linear subspace, $\bm \Delta^\circ(\Ss)$ contains at least the zero array, and either only that one or infinitely many other arrays as well. Denote $d_\circ$ the dimension of $\bm \Delta^\circ(\Ss)$. It appears from Condition~\ref{cond:Gamma:S}, \eqref{eq:gamma:q} and (\ref{eq:gamma:rep}) that, if $d_\circ = 0$ (i.e., $\bm \Delta^\circ(\Ss)$ contains only the zero array), then $\Gamma(\Ss)$ contains only $\ggamma^{(q,\Ss)}$. In such a case, parameterising $\Gamma(\Ss)$ appears vacuous. In contrast, if $d_\circ \geq 1$, there are infinitely many copula arrays in $\Gamma(\Ss)$. Therefore, to proceed, we will assume: 

\begin{condition} $d_\circ = \dim(\bm \Delta^\circ(\Ss)) \geq 1$. 
	\label{cond:Gamma:S:infinite} \end{condition}
\noindent Under Condition~\ref{cond:Gamma:S:infinite}, we may find $d_\circ$ arrays $\Delta_1, \dots, \Delta_{d_\circ} \in \bm \Delta^\circ(\Ss) \subset \R^{r_1 \times \cdots \times r_d}$ such that any $\ggamma$ in $\Gamma(\Ss)$ can be expressed as
\begin{equation*} \ggamma = \ggamma^{(q,\Ss)} + \sum_{k=1}^{d_\circ}\theta_k \Delta_k,  \end{equation*}
in terms of a unique vector $\bm \theta = (\theta_1, \dots, \theta_{d_\circ}) \in \R^{d_\circ}$. Incidentally, this means that $\Gamma(S)$ is a subset of an {\it affine space} with $\Delta_1, \dots, \Delta_{d_\circ}$ as basis; see \citet[Chapter 2]{Berger09}. Given this basis, any copula array of $\Gamma(\Ss)$ is unequivocally identified by its coefficients  $ (\theta_1, \dots, \theta_{d_\circ})$. We may, therefore, represent arrays of $\Gamma(\Ss)$ as $\ggamma \doteq \ggamma(\ttheta)$, where $\ggamma: \R^{d_\circ} \to \R^{r_1 \times \cdots \times r_d}$ is the function
\begin{equation} \ttheta = (\theta_1,\ldots,\theta_{d_\circ}) \ \mapsto \  \ggamma(\bm \theta) = \ggamma^{(q,\Ss)} + \sum_{k=1}^{d_\circ}\theta_k \Delta_k. 	\label{eq:gamma:decomp}  \end{equation}
Identifying $\ggamma = \ggamma(\ttheta)$ means that $\ggamma$ is sometimes understood as a copula array and sometimes as the function of $\ttheta$ returning that copula array. The context clarifies what is intended.

\subsubsection{Practical identification of a basis} \label{subsec:basis}

It will be necessary here to vectorise the multi-dimensional arrays under consideration, so that we may apply standard operations on vectors and matrices. While this can produce somewhat unwieldy expressions, their complexity is merely notational and the underlying ideas remain entirely grounded in elementary principles of linear algebra. Let $\vect(\cdot)$ denote the $d$-dimensional extension of the column-major vectorisation operator which transforms an array in $\R^{r_1 \times \cdots \times r_d}$ into a vector of $\R^{\prod_{\ell = 1}^d r_\ell}$. Specifically, given an array $\ab \in  \R^{r_1 \times \cdots \times r_d}$, $\vect(\ab)$ is a vector of $\R^{\prod_{\ell = 1}^d r_\ell}$ whose elements are given by $\vect{(\ab)}_{\phi(\bm i)} = a_{\ii}$ ($\ii \in \Rs$), where $\phi:\Rs \to \left[\prod_{\ell = 1}^d r_\ell \right]$ is
\begin{equation}
	\label{eq:phi}
	\phi(\bm i) = i_1 + \sum_{\ell_1=2}^{d} \left( (i_{\ell_1} - 1) \prod_{\ell_2=1}^{\ell_1-1} r_{\ell_2} \right), \qquad \ii = (i_1,\ldots,i_d) \in \Rs.
\end{equation}
The inverse mapping $\vect^{-1}: \R^{\prod_{\ell = 1}^d r_\ell } \to \R^{r_1 \times \cdots \times r_d}$ is such that, for any starting vector $\bm v \in \R^{\prod_{\ell = 1}^d r_\ell }$, the elements of the resulting array $\vect^{-1}(\bm v) \in \R^{r_1 \times \cdots \times r_d}$ are given by
\begin{equation*} \vect^{-1}(\bm v)_{\phi^{-1}(j)} = v_j, \qquad j \in \left[\prod_{\ell = 1}^d r_\ell \right], \label{eq:vect:inv} \end{equation*}
where $\phi^{-1}(j) = {\bm i} = (i_1, i_2, \dots, i_d)$ is computed as follows: initialise $k \leftarrow j - 1$; for $\ell \in  [d]$: set $i_\ell = (k \bmod r_\ell) + 1$ and update $k \leftarrow \left\lfloor \frac{k}{r_\ell} \right\rfloor$. 

\ppn Throughout, we will denote $\un{\bm p} \doteq \vect(\pp) \in [0,1]^{\prod_{\ell = 1}^d r_\ell }$ and $\un{\ggamma} \doteq \vect(\ggamma) \in [0,1]^{\prod_{\ell = 1}^d r_\ell }$, and more generally the vectorised versions of $\Ps(\Ss)$ and $\Gamma(\Ss)$ in (\ref{eq:Gamma:S}) as
\begin{align*} \un{\Ps}(\Ss) & \doteq \{  \un{\bm p} \in [0,1]^{\prod_{\ell = 1}^d r_\ell} : \un{\bm p} = \vect(\pp), \pp \in \Ps(\Ss) \},  \\
	\un{\Gamma}(\Ss) &  \doteq \{ \un{\ggamma} \in [0,1]^{\prod_{\ell = 1}^d r_\ell} : \un{\ggamma} = \vect(\bm \gamma), \bm \gamma \in \Gamma(\Ss) \}.  
\end{align*}
Likewise, from $\bm \Delta^\circ(\Ss)$ in (\ref{eq:DDelta}), let 
\begin{equation} \un{\bm \Delta}^\circ(\Ss) = \{ \bm x \in \R^{ \prod_{\ell = 1}^d r_\ell }: \bm x = \vect(\Delta), \Delta \in \bm \Delta^\circ(\Ss) \}. \label{eq:DDelta:vect} \end{equation}
The linear and bijective nature of $\vect(\cdot):\R^{r_1 \times \cdots \times r_d} \to \R^{\prod_{\ell = 1}^d r_\ell}$ has two implications. First, $\dim\left(\un{\bm \Delta}^\circ(\Ss)\right) = \dim\left(\bm \Delta^\circ(\Ss)\right) = d_\circ$; and second, the vectors $\un{\Delta}_1, \dots, \un{\Delta}_{d_\circ} \in \R^{ \prod_{\ell = 1}^d r_\ell }$ form a basis of $\un{\bm \Delta}^\circ(\Ss)$ if and only if the arrays $\vect^{-1}(\un{\Delta}_1), \dots, \vect^{-1}(\un{\Delta}_{d_\circ}) \in \R^{r_1 \times \cdots \times r_d}$ form a basis of $\bm \Delta^\circ(\Ss)$. Hence, to obtain a basis of $\bm \Delta^\circ(\Ss)$, it suffices to obtain a basis of $\un{\bm \Delta}^\circ(\Ss)$. We can proceed as follows.

\ppn As $\bm x \in \un{\bm \Delta}^\circ(\Ss) \iff \bm x = \vect(\Delta)$ for some $\Delta \in {\bm \Delta}^\circ(\Ss)$, any vector $\bm x \in \un{\bm \Delta}^\circ(\Ss)$ is necessarily such that $x_{\phi(\bm i)} = 0$ for all $\ii \in \Rs \setminus \Ss$ and 
\[ \sum_{ i_1 \in [r_1]} \dots \sum_{ i_{\ell-1} \in [r_{\ell-1}]} \sum_{ i_{\ell+1} \in [r_{\ell+1}]} \dots \sum_{ i_d \in [r_d]} x_{\phi(i_1, \dots, i_{\ell-1}, i,  i_{\ell+1}, \dots ,i_d)} = 0\]
for all $\ell \in [d]$ and $i \in [r_\ell]$. Let $C_\Ss$ be the $\left\{\left( \prod_{\ell = 1}^d r_\ell - | \Ss | + \sum_{\ell = 1}^d r_\ell \right) \times \prod_{\ell = 1}^d r_\ell\right\}$-matrix whose rows encode the above linear constraints, and consider the homogeneous system 
\begin{equation} C_\Ss \, \bm x = \bm 0. 	\label{eq:linsystem} \end{equation}
It appears that $\un{\bm \Delta}^\circ(\Ss)$ is the null space of the matrix $C_\Ss$, and $d_\circ$ is the dimension of that null space. Any algorithm computing a basis of the null space of a matrix can be used for identifying a basis for $\un{\bm \Delta}^\circ(\Ss)$. For example, the function \code{null(...)} from the \proglang{R} package \pkg{pracma} \citep{Borchers22} does this by exploiting the so-called QR-factorization of the input matrix \citep[pp.\, 49-50]{Trefethen97}. Indeed $\code{null}(C_\Ss)$ returns a matrix whose column vectors form a basis of the null space of $C_\Ss$. 

\begin{remark} \label{rmk:Cs} It is seen that Condition~\ref{cond:Gamma:S:infinite} holds if and only if the null space of $C_\Ss$ is non-trivial. A convenient practical test is to attempt to solve~\eqref{eq:linsystem}, at least numerically. If the system admits only the trivial solution (the zero vector), then Condition~\ref{cond:Gamma:S:infinite} does not hold. Otherwise, the existence of infinitely many solutions confirms that Condition~\ref{cond:Gamma:S:infinite} is satisfied.	\qed
\end{remark}

\begin{remark}\label{rmk:rect}
	The construction described above allows the identification of a basis of ${\bm \Delta}^\circ(\Ss)$ and/or $\un{\bm \Delta}^\circ(\Ss)$ for any support $\Ss\subseteq\Rs$, including supports with irregular patterns of zeros in $\pp\in\Ps(\Ss)$. In more structured contexts, however, specific natural parameterisations may emerge for $\Gamma(\Ss)$. In particular, for $\Ss = \Rs$ (hyper-rectangular support), $\DDelta^\circ(\Rs)$ is the linear subspace of $\R^{r_1\times\cdots\times r_d}$ comprising arrays with zero univariate margins; a convenient basis of which can be defined as follows. Let $\mathbf{1}_{r_\ell}\in\R^{r_\ell}$ be the vector of ones, $e_j^{(\ell)}$ denote the $j$th standard basis vector of $\R^{r_\ell}$, and set $c_{\ell j}=e_j^{(\ell)}-e_{r_\ell}^{(\ell)}$ for $j\in[r_\ell-1]$. For each $J\subseteq[d]$ with $|J|\geq2$ and each index tuple $\mathbf j_J=(j_\ell)_{\ell\in J}\in \bigtimes_{\ell \in J} [r_\ell-1]$, define the array $\Delta_{J,\mathbf j_J} \in \R^{r_1 \times \cdots \times r_d}$ by $(\Delta_{J,\mathbf j_J})_{i_1,\ldots,i_d}=\prod_{\ell=1}^d(b_\ell)_{i_\ell}$, where $b_\ell=c_{\ell j_\ell}$ if $\ell\in J$ and $b_\ell=\mathbf{1}_{r_\ell}$ otherwise. These arrays form a basis of $\DDelta^\circ(\Rs)$. Indeed, allowing all $J\subseteq[d]$, including $|J|=0$ and $|J|=1$, gives a basis of the full array space, since $\{\mathbf{1}_{r_\ell},c_{\ell1},\ldots,c_{\ell,r_\ell-1}\}$ is a basis of $\R^{r_\ell}$ for every $\ell$. The arrays with $|J|\geq2$ have zero univariate margins. Conversely, if an array has zero univariate margins, its coefficient for $J=\varnothing$ must vanish by its total sum, and its coefficients for $J=\{\ell\}$ must vanish by its $\ell$th univariate margin. It follows that, in the hyper-rectangular support case,
	\[ 	d_\circ=\sum_{\substack{J\subseteq[d]\\|J|\geq2}}\prod_{\ell\in J}(r_\ell-1)=\prod_{\ell=1}^{d}r_\ell-1-\sum_{\ell=1}^{d}(r_\ell-1). \]
	For example, for the $(3 \times 3)$-matrices $(d=2$) of full supports, the dimension of $\DDelta^\circ(\Rs)$ is $d_\circ = 4$ and  writing $\Delta_{(j_1,j_2)}$ for $\Delta_{[2],(j_1,j_2)}$, we obtain the following natural basis of $\DDelta^\circ(\Rs)$:
	\begin{equation} \Delta_{(1,1)} = \begin{pmatrix} 1 & 0 & -1 \\ 0 & 0 & 0 \\ -1 & 0 & 1\end{pmatrix}, \Delta_{(1,2)} = \begin{pmatrix} 0 & 1 & -1 \\ 0 & 0 & 0 \\ 0 & -1 & 1\end{pmatrix}, \Delta_{(2,1)} = \begin{pmatrix} 0 & 0 & 0 \\ 1 & 0 & -1 \\ -1 & 0 & 1\end{pmatrix}, \Delta_{(2,2)} = \begin{pmatrix} 0 & 0 & 0 \\ 0 & 1 & -1 \\ 0 & -1 & 1\end{pmatrix}.  \label{eqn:canon33} \end{equation} \qed
\end{remark}

\subsubsection{Resulting parameterisation} \label{subsec:param} 

Assume Conditions \ref{cond:Gamma:S}-\ref{cond:Gamma:S:infinite} and suppose that a basis $\Delta_1, \dots, \Delta_{d_\circ}$ of $\bm \Delta^\circ(\Ss)$ has been identified as suggested above (or otherwise). Consider the vectorised version of $\ggamma(\ttheta)$ in \eqref{eq:gamma:decomp}; viz.\ the function $\un{\ggamma}: \R^{d_\circ} \to \R^{\prod_{\ell = 1}^d r_\ell}$ defined as $\un{\ggamma}(\bm \theta)  \doteq \vect(\ggamma(\ttheta))$, or explicitly
\begin{equation}  \un{\ggamma}(\bm \theta) = \vect \left( \ggamma^{(q,\Ss)} + \sum_{k = 1}^{d_\circ} \theta_{k} \Delta_k \right) = \un{\ggamma}^{(q,\Ss)} + \sum_{k = 1}^{d_\circ} \theta_{k} \un{\Delta}_k = \un{\ggamma}^{(q,\Ss)} + A_\Ss\bm \theta, 	\label{eq:un:G:vartheta}  \end{equation}
with $A_\Ss$ the $\left(\prod_{\ell = 1}^d r_\ell  \times d_\circ\right)$-matrix whose columns vectors are $\un{\Delta}_1, \dots, \un{\Delta}_{d_\circ}$. (This is exactly the outcome of the \code{null(...)} function of the \pkg{pracma} R package mentioned at the end of Section \ref{subsec:basis}.) Reversely, we may rewrite \eqref{eq:gamma:decomp} as
\begin{equation} \ggamma(\bm \theta) = \vect^{-1} (\un{\ggamma}(\bm \theta)) = \ggamma^{(q,\Ss)} + \vect^{-1}(A_\Ss\bm \theta). 	\label{eq:G:vartheta} \end{equation}
Now, for any $\ttheta \in \R^{d_\circ}$, this array $\ggamma(\ttheta)$ has support $\Ss$ and discrete uniform margins ($[\ggamma(\bm \theta)]_i^{(\ell)} = 1/r_{\ell} \ \forall\, i \in [r_\ell]$ and $\ell \in [d]$) by construction. However, it is not necessarily the case that $0 < [\ggamma(\ttheta)]_{\ii} < 1$ for all $\ii \in \Ss$; i.e., without restriction on $\ttheta$, $\ggamma(\ttheta)$ is not necessarily a valid copula array of $\Gamma(\Ss)$.

\ppn Let $R_\Ss$ be the $(|\Ss| \times \prod_{\ell = 1}^d r_\ell)$-matrix obtained from the $(\prod_{\ell = 1}^d r_\ell \times \prod_{\ell = 1}^d r_\ell)$-identity matrix by deleting its rows with index $\phi(\bm i)$, $\bm i \in \Rs \setminus \Ss$, where $\phi$ is defined in~\eqref{eq:phi}. See that $R_\Ss$ is an `entry selection matrix'; that is, left-multiplying by $R_\Ss$ allows us to, roughly speaking, `retain only the elements of vectorised arrays that are in the support $\Ss$' while multiplying by its transpose performs the inverse operation, which amounts to `adding zeros at the right spots' in vectors. Call $\Theta \subset \R^{d_\circ}$ the admissible parameter set; that is, $\ttheta \in \Theta \iff \ggamma(\ttheta) \in \Gamma(\Ss) \iff 0 < [\ggamma(\ttheta)]_{\ii} < 1$ for all $\bm i \in \Ss$. The latter set of inequalities can be written under matrix form as $\bm 0_{\R^{|\Ss|}} < R_\Ss (\un{\ggamma}^{(q,\Ss)} + A_\Ss\bm \theta) < \bm 1_{\R^{|\Ss|}}$. Hence,
\begin{equation} \Theta = \left\{ \bm \theta \in \R^{d_\circ} : - R_\Ss\un{\ggamma}^{(q,\Ss)}  < \widetilde{A}_\Ss\bm \theta < \bm 1_{\R^{|\Ss|}} - R_\Ss\un{\ggamma}^{(q,\Ss)} \right\}, \label{eq:Theta} \end{equation}
where $\widetilde{A}_\Ss \doteq R_\Ss A_\Ss$, the $(|\Ss| \times d_\circ)$-matrix obtained from $A_\Ss$ by deleting its null rows (corresponding to the structurally null entries of all arrays of $\Delta^\circ(\Ss)$). Note that $R_\Ss\un{\ggamma}^{(q,\Ss)}$ is the vector listing the non-null entries of $\ggamma^{(q,\Ss)}$. In particular, for $\Ss = \Rs$ (hyper-rectangular support), $R_\Ss = I$ the $(\prod_{\ell = 1}^d r_\ell \times \prod_{\ell = 1}^d r_\ell)$-identity matrix, and $\widetilde{A}_\Rs = A_\Rs$ is a matrix containing a (vectorised version of a) basis of $\DDelta^\circ(\Rs)$, for example as in Remark \ref{rmk:rect}. Now, as a relatively obvious consequence of the construction, we have:

\begin{lemma} \label{lem:bij} Under Conditions \ref{cond:Gamma:S}-\ref{cond:Gamma:S:infinite}, the function $\ttheta \mapsto \un{\ggamma}(\ttheta) $ in~\eqref{eq:un:G:vartheta} is an affine bijection between $\Theta$ and $\un{\Gamma}(\Ss)$, and the function $\ttheta \mapsto \ggamma(\ttheta) $ in~\eqref{eq:G:vartheta} is a bijection between $\Theta$ and $\Gamma(\Ss)$.
\end{lemma}

\begin{proof} See Appendix \ref{app:proofs} in Supplementary Material. \end{proof}

\noindent Given this one-to-one correspondence between $\Theta$ and $\Gamma(\Ss)$, the unique $\Is$-projection of $\pp \in \Ps(\Ss)$ onto $\Gamma(\Ss)$ can be characterised in terms of its vector $\ttheta$. Specifically, if $\ggamma_\pp$ is the $\Is$-projection (\ref{eqn:copproj}) of $\pp$ on $\Gamma(\Ss)$, then we have $\ggamma_\pp = \ggamma(\ttheta_\pp)$ where
\begin{equation} \ttheta_\pp  \doteq \arg \min_{\ttheta \in \Theta} \Is(\ggamma(\ttheta) \| \pp). \label{eqn:thetastar}\end{equation}
Under Conditions \ref{cond:Gamma:S}-\ref{cond:Gamma:S:infinite}, such $\ttheta_\pp$ exists and is unique in $\Theta$, as $\ggamma_\pp$ exists and is unique in $\Gamma(\Ss)$. All this is illustrated below on two simple examples: one bivariate ($d=2$) and one trivariate ($d=3$).

\begin{example} \label{ex:example} Let $\pp$ be a $(3 \times 3)$ probability array with $\Ss_\pp = \{(i_1,i_2) \in [3]^2 = \{1,2,3\} \times \{1,2,3\} \text{ s.t.\ } i_1 \neq i_2\}$, viz.
	\begin{equation} \pp =\begin{pmatrix} 0 & p_{12} & p_{13} \\ p_{21} & 0 & p_{23} \\ p_{31} & p_{32} & 0 \end{pmatrix} \label{eqn:pxyzeros}\end{equation}
	for some positive probabilities $p_{12}, p_{13}, p_{21}, p_{23}, p_{31}, p_{32}$ summing to 1. On this support $\Ss_\pp$, the quasi-independence copula array is
	\[\ggamma^{(q,\Ss_\pp)}=\begin{pmatrix} 0 & 1/6 & 1/6 \\ 1/6 & 0 & 1/6 \\ 1/6 & 1/6 & 0 \end{pmatrix}, \]
	which is in fact directly the array defined by (\ref{eqn:qquasi}). Each row and column of any array of $\DDelta^\circ(\Ss_\pp)$ contains at most two non-zero entries, hence any single positive value propagates through the matrix via the zero row and column sum constraints, thereby uniquely determining all remaining entries. The set $\DDelta^\circ(\Ss_\pp)$ is thus unidimensional ($d_\circ = 1$), and we can take for its basis element:
	\[\Delta_1 \doteq \begin{pmatrix} 0 & 1 & -1 \\ -1 & 0 & 1 \\ 1 & -1 & 0 \end{pmatrix}. \]
	According to (\ref{eq:gamma:decomp}), any copula array in $\Gamma(\Ss_\pp)$ can be written as:
	\[\ggamma(\theta_1) = \ggamma^{(q,\Ss_\pp)} + \theta_1 \Delta_1 = \begin{pmatrix} 0 & 1/6 + \theta_1 & 1/6 - \theta_1 \\ 1/6 - \theta_1 & 0 & 1/6 + \theta_1 \\ 1/6 + \theta_1 & 1/6 - \theta_1 & 0 \end{pmatrix}, \]
	where $\theta_1$ is restricted to  $\Theta = (-1/6,1/6)$. Some basic algebra reveals that 
	\[\Is(\ggamma(\theta_1)||\pp) = \sum_{(i,j) \in \Ss_\pp} \gamma_{ij}(\theta_1) \log \frac{\gamma_{ij}(\theta_1)}{p_{ij}} = \left( \frac{1}{6} + \theta_1 \right) \log \frac{\left( \frac{1}{6} + \theta_1 \right)^3}{p_{12} p_{23} p_{31}} + \left( \frac{1}{6} - \theta_1 \right) \log \frac{\left( \frac{1}{6} - \theta_1 \right)^3}{p_{13} p_{21} p_{32}}\]
	admits a unique minimum for 
	\[\theta_{\pp;1} = \frac{1}{6}\, \frac{{\omega_\pp}^{1/3}-1}{{\omega_\pp}^{1/3}+1}, \qquad \text{ with } \omega_\pp =\frac{p_{13}p_{21}p_{32}}{p_{12}p_{23}p_{31}}. \]
	The copula array $\ggamma_\pp$ of $\pp$ in (\ref{eqn:pxyzeros}) is thus 
	\[ \ggamma_\pp= \ggamma(\theta_{\pp;1}) = \begin{pmatrix} 0 & \frac{1}{3}\frac{{\omega_\pp}^{1/3}}{1+{\omega_\pp}^{1/3}} &  \frac{1}{3}\frac{1}{1+{\omega_\pp}^{1/3}} \\ \frac{1}{3}\frac{1}{1+{\omega_\pp}^{1/3}} & 0 &  \frac{1}{3}\frac{{\omega_\pp}^{1/3}}{1+{\omega_\pp}^{1/3}} \\  \frac{1}{3}\frac{{\omega_\pp}^{1/3}}{1+{\omega_\pp}^{1/3}} & \frac{1}{3}\frac{1}{1+{\omega_\pp}^{1/3}} & 0 \end{pmatrix}.\]
	This shows that the dependence of any distribution of type (\ref{eqn:pxyzeros}) is entirely described by the `odds-ratio-like' quantity $\omega_\pp$. In particular, quasi-independence corresponding to the value $\theta_1 = 0$, we conclude that a probability array $\pp$ as in (\ref{eqn:pxyzeros}) reflects quasi-independence when $\omega_\pp =1$ or equivalently $p_{13}p_{21}p_{32}=p_{12}p_{23}p_{31}$, confirming \citet[equation (2.23)]{Goodman68}. \qed
\end{example}

\begin{example}  \label{ex:example2} Consider a trivariate $(d=3$) discrete vector $(X_1,X_2,X_3) \sim P$ with $r_1 = 4$, $r_2 =2$ and $r_3 =2 $, and probability array $\pp$ of the form:
	\begin{equation}  \pp =
		\left\{
		\begin{bmatrix}
			p_{111} & p_{121} \\
			0 & 0 \\
			p_{311} & p_{321} \\
			p_{411} & p_{421}
		\end{bmatrix}, 
		\begin{bmatrix}
			p_{112} & p_{122} \\
			p_{212} & p_{222} \\
			p_{312} & p_{322} \\
			p_{412} & p_{422}
		\end{bmatrix}
		\right\}. \label{eqn:pex2}
	\end{equation}
	Here the array $\pp$ is represented as a collection of \( r_3 = 2 \) matrices of size \( (r_1 \times r_2)=(4 \times 2) \) -- each matrix corresponds to a fixed level of \( X_3 \), with entries \( p_{i_1 i_2 i_3} = \mathbb{P}(X_1 = i_1, X_2 = i_2, X_3 = i_3) \) arranged in standard matrix form. It is assumed that $p_{i_1 i_2 i_3} > 0$ for all $(i_1,i_2,i_3 ) \in \Rs = [4] \times [2] \times [2]$, except for $p_{211} = p_{221} = 0$. For this specific support $\Ss_\pp$, the quasi-independence copula array is
	\begin{equation}  \ggamma^{(q,\Ss_\pp)} =
		\left\{
		\begin{bmatrix}
			1/12 & 1/12 \\
			0 & 0 \\
			1/12 & 1/12 \\
			1/12 & 1/12 
		\end{bmatrix}, 
		\begin{bmatrix}
			1/24 & 1/24 \\
			1/8 & 1/8 \\
			1/24 & 1/24 \\
			1/24 & 1/24 
		\end{bmatrix}
		\right\}, \label{eqn:coparrayqSp}
	\end{equation}
	which is the $\Is$-projection of the `quasi-uniform' array $\qq$ \eqref{eqn:qquasi} (that is, the probability array \eqref{eqn:pex2} with all non-zero entries equal to $1/14$) onto the relevant copula Fr\'echet class (\ref{eq:Gamma}); viz.
	\begin{equation} \Gamma = \Fs({\bm u}^{(1)},{\bm u}^{(2)},{\bm u}^{(3)}) = \left\{ \ggamma \in \Ps : \gamma^{(1)}_i = \frac{1}{4} \text{ for all } i \in [4], \gamma^{(2)}_i = \gamma^{(3)}_i = \frac{1}{2} \text{ for all } i \in [2] \right\}. \label{eqn:Gamma422} \end{equation}
	In this configuration, with two structural zeros and 8 marginal constraints, the 0-1 matrix $C_{\Ss_\pp}$ in (\ref{eq:linsystem}) has 10 rows for $r_1 r_2 r_3 =16$ columns, and can easily be written explicitly (Appendix \ref{app:ex2}). As suggested in Section \ref{subsec:basis}, we used the R procedure {\tt pracma::null} to identify the null space of $C_{\Ss_\pp}$, which revealed its dimension $d_\circ = 8$ (in particular, Condition \ref{cond:Gamma:S:infinite} is fulfilled for the support $\Ss_\pp$). From (\ref{eq:G:vartheta}), it appears that the copula array of any probability array $\pp$ like (\ref{eqn:pex2}) can be written as 
	\[\ggamma_\pp =  \ggamma^{(q,\Ss_\pp)} + \vect^{-1}(A_{\Ss_\pp} \ttheta), \]
	for some unique $\ttheta \in \Theta \subset \R^8$ and a given $(16 \times 8)$-matrix  $A_{\Ss_\pp}$ encoding a basis for the null space of $C_{\Ss_\pp}$ -- the procedure {\tt pracma::null} returned such a matrix, see Appendix \ref{app:ex2}. Each $\ttheta \in \Theta$ unequivocally represents a specific dependence structure compatible with the support of $\pp$ in (\ref{eqn:pex2}). In particular, $ \ttheta = \zero$ captures the quasi-independence of $(X_1,X_2,X_3)$; see Section \ref{subsec:quasitest}. \qed 
\end{example}

\subsubsection{Differentiability results}

Consider now the analogue to (\ref{eqn:thetastar}) seen as a function from $\un{\Ps}(\Ss)$ to $\Theta$; viz.
\begin{equation} \ttheta^*(\un{\bm p}) \doteq \arg \min_{\bm \theta \in \Theta} \Is(\ggamma(\bm \theta) \| \vect^{-1}(\un{\bm p})). 
	\label{eq:vartheta:*}  \end{equation}
In other words, $ \ttheta^*(\un{\bm p})$ is the (unique) vector $\ttheta$ identifying the $\Is$-projection onto $\Gamma(\Ss)$ of the probability array whose vectorisation is $\un{\bm p}$; i.e., $\ttheta^*(\vect(\pp)) = \ttheta_\pp$.
In the same vein, let $\un{\ggamma}^*: \un{\Ps}(\Ss) \to \un{\Gamma}(\Ss)$:
\begin{equation} \un{\ggamma}^*(\un{\bm p}) = \un{\ggamma} (\ttheta^*(\un{\bm p})), 	\label{eq:gamma:*} \end{equation}
where $\un{\ggamma}$ is given in (\ref{eq:un:G:vartheta}), and which returns the vectorised copula array of the probability array whose vectorisation is $\un{\bm p}$; i.e.\ $\un{\ggamma}^*(\vect(\pp)) = \un{\ggamma}(\ttheta_\pp)$ and thus  $\vect^{-1}(\un{\ggamma}^*(\vect(\pp))) = \ggamma(\ttheta_\pp) = \ggamma_\pp$, the copula array of $\pp$. The following result -- essentially a consequence of \citet[Corollary 31]{Geenens24-diff} -- establishes that $\ttheta^*$ is a continuously differentiable function of $\un{\bm p}$ over the whole of $\un{\Ps}(\Ss)$, and provides its Jacobian.

\begin{proposition} \label{prop:diff} Assume Conditions \ref{cond:Gamma:S}-\ref{cond:Gamma:S:infinite}. Then, 
	the function $\ttheta^*$ (\ref{eq:vartheta:*}) is continuously differentiable at any $\un{\bm p} \in \un{\Ps}(\Ss)$, and its Jacobian matrix is 
	\begin{equation} J_{\ttheta^*}(\un{\bm p}) = \left[\transp{\widetilde{A}_\Ss} \diag^{-1}\left( R_\Ss \un{\ggamma}^*(\un{\bm p})\right) \widetilde{A}_\Ss\right]^{-1} \transp{\widetilde{A}_\Ss} \diag^{-1}\left( R_\Ss \un{\bm p}\right) R_\Ss \label{eqn:jactheta} \end{equation}
	with the matrices $R_\Ss$, $A_\Ss$ and $\widetilde{A}_\Ss = R_\Ss  A_\Ss$ defined in Section \ref{subsec:param}.
\end{proposition}
\begin{proof} See Appendix \ref{app:proofs} in Supplementary Material. \end{proof} 

\ppn In addition, see that $\un{\ggamma}^*(\un{\bm p}) = \un{\ggamma}^{(q,\Ss)} + A_\Ss \ttheta^*(\un{\bm p})$, which is a simple re-arrangement of (\ref{eq:un:G:vartheta}) and (\ref{eq:gamma:*}). It directly follows from this and Proposition \ref{prop:diff}:
\begin{proposition} \label{prop:diffgamma} Assume Conditions \ref{cond:Gamma:S}-\ref{cond:Gamma:S:infinite}. Then, 
	the function $\un{\ggamma}^*$ (\ref{eq:gamma:*}) is continuously differentiable at any $\un{\bm p} \in \un{\Ps}(\Ss)$, and its Jacobian matrix is 
	\begin{equation} J_{\un{\ggamma}^*}(\un{\bm p}) = A_\Ss J_{\ttheta^*}(\un{\bm p})  =  A_\Ss \left[\transp{\widetilde{A}_\Ss} \diag^{-1}\left( R_\Ss \un{\ggamma}^*(\un{\bm p})\right) \widetilde{A}_\Ss\right]^{-1} \transp{\widetilde{A}_\Ss} \diag^{-1}\left( R_\Ss \un{\bm p}\right) R_\Ss, \label{eq:Jacobian:gamma:*} \end{equation}
	with the matrices $R_\Ss$, $A_\Ss$ and $\widetilde{A}_\Ss = R_\Ss  A_\Ss$ defined in Section \ref{subsec:param}.
\end{proposition}

\begin{remark}  \label{rem:invbasis}   The $(\prod_{\ell = 1}^d r_\ell \times \prod_{\ell = 1}^d r_\ell)$-matrix $J_{\un{\ggamma}^*}(\un{\bm p})$  is invariant under a change of basis for $\DDelta^\circ(\Ss)$. Indeed, let $\un{\Delta}_1', \dots, \un{\Delta}_{d_\circ}'$ be another basis of $\un{\bm \Delta}^\circ(\Ss)$. Call $A_\Ss'$ the $\left((\prod_{\ell = 1}^d r_\ell) \times d_\circ\right)$- matrix whose columns are $\un{\Delta}_1', \dots, \un{\Delta}_{d_\circ}'$. Then, there exists a $(d_\circ \times d_\circ)$-invertible matrix $B$ such that $ A_\Ss = A'_\Ss B$; whereby
	\begin{align*} J_{\un{\ggamma}^*}(\un{\bm p}) & = A'_\Ss B \left[\transp{B}\transp{\widetilde{A'}_\Ss} \diag^{-1}\left( R_\Ss \un{\ggamma}^*(\un{\bm p})\right) \widetilde{A'}_\Ss B \right]^{-1} \transp{B}\transp{\widetilde{A'}_\Ss} \diag^{-1}\left( R_\Ss \un{\bm p}\right) R_\Ss \\
		& = A'_\Ss B B^{-1}\left[\transp{\widetilde{A'}_\Ss} \diag^{-1}\left( R_\Ss \un{\ggamma}^*(\un{\bm p})\right) \widetilde{A'}_\Ss  \right]^{-1} [\transp{B}]^{-1}\transp{B}\transp{\widetilde{A'}_\Ss} \diag^{-1}\left( R_\Ss \un{\bm p}\right) R_\Ss \\
		& = A'_\Ss \left[\transp{\widetilde{A'}_\Ss} \diag^{-1}\left( R_\Ss \un{\ggamma}^*(\un{\bm p})\right) \widetilde{A'}_\Ss  \right]^{-1} \transp{\widetilde{A'}_\Ss} \diag^{-1}\left( R_\Ss \un{\bm p}\right) R_\Ss,
	\end{align*} 
	where $\widetilde{A'}_\Ss =  R_\Ss A'_\Ss$. We recover (\ref{eq:Jacobian:gamma:*}) where $A_\Ss$ is simply replaced by $A'_\Ss$. \qed
\end{remark} 

\section{The empirical discrete copula} \label{sec:empcop}

Suppose that we have $n$ copies $(\XX_1,\dots,\XX_n)$ of a $d$-dimensional discrete random vector $\XX = (X_1,\ldots,X_d) \in \Xs$ drawn from an unknown distribution $P$, whose probability array $\pp$ belongs to $\Ps(\Ss_\pp)$ for some support $\Ss_\pp \subseteq \Rs$ defined as in (\ref{eq:Rs}). It will be assumed that $\Ss_\pp$ is known: this is usually the case in practice, as entries $p_{\ii} = 0$ typically arise from strict physical incompatibilities between some specific values of $X_1,\ldots,X_d$ and these are normally well identified -- \cite{Goodman68} calls them `{\it a priori considerations}'. We can then verify that $\Gamma(\Ss_\pp) \neq \emptyset$ (Condition \ref{cond:Gamma:S}) -- at least numerically, see Remark \ref{rmk:checkcond} -- which ensures the existence and uniqueness of the copula array $\ggamma_\pp$ and thus the copula measure $G_\pp$ of $P$ (Proposition \ref{prop:sklar:like}/Definition \ref{def:cop}). If that is the case, we would like to estimate $\ggamma_\pp$ from a realisation of $(\XX_1,\dots,\XX_n)$, as a margin-free representation of the inner dependence structure of $\XX$.

\ppn A natural estimator of $\pp$ is the empirical probability array $\pp_n$ defined entrywise by 
\begin{equation} p_{\ii;n} = \frac{1}{n} \sum_{k=1}^n \indic{\XX_k =  (x_{i_1}^{(1)}, \dots, x_{i_d}^{(d)})}, \qquad \bm i = (i_1, i_2, \dots, i_d) \in \Rs.   \label{eq:pn} \end{equation}
If $\pp_n$ is a strongly consistent estimator of $\pp$ (for example, when $(\XX_1,\dots,\XX_n)$ is a random sample from $P$), we have $\pp_n \in \Ps(\Ss_\pp)$ with probability 1 as $n \to \infty$. However, in finite samples, it may be the case that $p_{\ii;n} = 0$ for some $\ii \in \Rs$ such that $p_{\ii} > 0$ due to random sampling. In other words, in practice there is no guarantee that $\pp_n \in \Ps(\Ss_\pp)$ or even $\pp_n \in \Ps$ (when $p_{i;n}^{(\ell)} = 0$ for some $\ell \in [d]$ and $i \in [r_\ell]$). If $\min_{\ii \in \Ss_\pp} p_{\ii;n} = 0$, we may want to use some `smoothed' version of $\pp_n$ which would slightly spread the observed probability mass over the empty cells of $\Ss_\pp$. For example, we may use the estimator
\begin{equation} \widehat{\pp}_n = \frac{n}{n+1} \pp_n + \frac{1}{n+1} \qq, \label{eqn:hatpn} \end{equation}
which is a $(\frac{n}{n+1},\frac{1}{n+1})$-mixture between the empirical probability array (\ref{eq:pn}) and $\qq \in \Ps(\Ss_\pp)$ the probability array defined by (\ref{eqn:qquasi}) -- other approaches may be considered, for example see \cite{Fienberg73}. As the resulting $\widehat{\pp}_n$ is guaranteed to belong to $\Ps(\Ss_\pp)$ for all $n \geq 1$, under the condition $\Gamma(\Ss_\pp)\neq \emptyset$ it admits a unique $\Is$-projection on $\Gamma$ (Proposition \ref{prop:sklar:like}). We define:

\begin{definition} Let $(\XX_1,\dots,\XX_n)$ be a sample drawn from $P$, where the support $\Ss_\pp$ of the probability array $\pp$ of $P$ is such that $\Gamma(\Ss_\pp)\neq \emptyset$. Then, the copula array of $\widehat{\pp}_n$,viz. 
	\begin{equation} \widehat{\ggamma}_n \doteq \Is_\Gamma(\widehat{\pp}_n) = \arg \min_{\ggamma \in \Gamma} \Is(\ggamma\|\widehat{\pp}_n) , \label{eqn:empcoparray}\end{equation}
	is called the empirical copula array of the sample, and the probability measure $\widehat{G}_n$ supported on $\bigtimes_{\ell \in [d]} \frac{[r_\ell]}{r_\ell + 1}$ with probability array $\widehat{\ggamma}_n$ is called the empirical discrete copula measure of the sample. In turn, the associated cumulative distribution function and probability mass function will be called the empirical discrete copula and empirical copula pmf of the sample.
\end{definition}

\begin{remark}  If the raw empirical cell probabilities (\ref{eq:pn}) are all positive on $\Ss_\pp$ (i.e., $\min_{\ii \in \Ss_\pp} p_{\ii;n} > 0$), then one may safely take $\widehat{\pp}_n = \pp_n$ and all asymptotic results below go through unchanged. However, for uniformity of exposition, and since a small amount of smoothing often improves finite‐sample performance \citep{Fienberg73}, we shall continue to define the empirical discrete copula array as in (\ref{eqn:empcoparray}) -- i.e., in terms of the mixture estimator \eqref{eqn:hatpn} -- even when $\pp_n \in \Ps(\Ss_\pp)$. \qed
\end{remark} 

\noindent Assume Condition \ref{cond:Gamma:S:infinite} is fulfilled for $\Ss_\pp$ and suppose that a basis $\Delta_1, \dots, \Delta_{d_\circ}$ of $\bm \Delta^\circ(\Ss_\pp)$ has been identified, defining a basis matrix $A_{\Ss_\pp}$ and a parameter space $\Theta$ as in Section \ref{subsec:param}. As $\widehat{\ggamma}_n \in \Gamma(\Ss_\pp)$, it follows from Lemma \ref{lem:bij} that there exists a unique $\widehat{\ttheta}_n \in \Theta$ such that 
\begin{equation} \widehat{\ggamma}_n = \ggamma(\widehat{\ttheta}_n ), \label{eqn:hatgamma} \end{equation} 
and by definition (\ref{eqn:empcoparray}), 
\begin{equation} \widehat{\ttheta}_n = \arg \min_{\ttheta \in \Theta} \Is(\ggamma(\ttheta) \| \widehat{\pp}_n). \label{eqn:hattheta}  \end{equation}
As $\dim(\Theta) = d_\circ < \infty$, we may regard $\{\ggamma(\ttheta): \ttheta \in \Theta\}$ as a parametric model which we try to fit to the sample $(\XX_1,\dots,\XX_n)$. From this perspective, (\ref{eqn:hatgamma})-(\ref{eqn:hattheta}) is a `minimum divergence estimator' (MDE) \citep{Basu11}. Of course, in general the sample was {\it not} generated by a discrete copula measure (in particular: $\pp \not\in \Gamma$), hence the `model' $\ggamma(\ttheta)$ is (knowingly) misspecified. Such a problem of MDE under model misspecification may be attacked via the approach of \cite{Jimenez11}, which was revisited in \cite{Geenens24-diff}. Through Proposition \ref{prop:diff}, we can establish the asymptotic properties of $\widehat{\ttheta}_n$ in (\ref{eqn:hattheta}).

\begin{proposition} \label{prop:thetahat} Let $(\XX_1,\dots,\XX_n)$ be a sample drawn from $P$, where the support $\Ss_\pp$ of the probability array $\pp$ of $P$ satisfies Conditions \ref{cond:Gamma:S} and \ref{cond:Gamma:S:infinite}. Let $\widehat{\un{\bm p}}_n \doteq \vect(\widehat{\pp}_n)$ and $\un{\bm \ggamma}_\pp = \vect(\ggamma_\pp)$. 
	\begin{itemize}		
		\item[(a)] If the estimator (\ref{eqn:hatpn}) is a (strongly) consistent estimator of $\pp$; that is, $\widehat{\pp}_n  \overset{\text{P}}{\to} \pp$ ($\widehat{\pp}_n  \overset{\text{as}}{\to} \pp$) as $n \to \infty$, then the estimator $\widehat{\ttheta}_n$ in (\ref{eqn:hattheta}) is a (strongly) consistent estimator of $\ttheta_\pp$ in (\ref{eqn:thetastar}); that is,
		\[ \widehat{\ttheta}_n \overset{\text{P}}{\to} \ttheta_\pp \quad (\widehat{\ttheta}_n \overset{\text{as}}{\to} \ttheta_\pp)\qquad \text{ as } n \to \infty;\]
		\item[(b)] If $s_n\left(\un{\widehat{\bm p}}_n- \un{\bm p}\right) \toL \bm Z$ as $n \to \infty$, for some  $s_n \to \infty$ and a given $(\prod_{\ell = 1}^d r_\ell)$-dimensional random vector $\bm Z$; then:
		\[  s_n\left(\widehat{\ttheta}_n - \ttheta_\pp \right) \toL 	J_{\ttheta^*}(\un{\bm p}) \bm Z, \] 
		where $J_{\ttheta^*}(\un{\bm p})$ is the Jacobian matrix defined in (\ref{eqn:jactheta}) with $\Ss = \Ss_\pp$;
		\item[(c)] In particular, if $(\XX_1,\dots,\XX_n)$ is an i.i.d.\ sample drawn from $P$ so that
		\begin{equation} \sqrt{n}\,\left(\un{\widehat{\bm p}}_n- \un{\bm p}\right) \toL \Ns_{\prod_{\ell = 1}^d r_\ell}\left(\bm 0,\diag(\un{\bm p}) - \un{\bm p} \, \transp{\un{\bm p}}\right), \label{eqn:empprop} \end{equation}
		then
		\begin{equation*} \sqrt{n}\,\left(\widehat{\ttheta}_n - \ttheta_\pp \right) \toL \Ns_{d_\circ}\left(\bm 0, \bm \Sigma_\ttheta\right) \qquad \text{as } n \to \infty, \label{eqn:asymptnormtheta} \end{equation*}
		where 
		\begin{equation} \bm \Sigma_\ttheta = \left[\transp{\widetilde{A}_{\Ss_\pp}} \diag^{-1}\left( R_{\Ss_\pp} \un{\ggamma}_\pp\right) \widetilde{A}_{\Ss_\pp}\right]^{-1} \left[\transp{\widetilde{A}_{\Ss_\pp}} \diag^{-1}\left( R_\Ss \un{\bm p}\right) \widetilde{A}_{\Ss_\pp} \right]\left[\transp{\widetilde{A}_{\Ss_\pp}} \diag^{-1}\left( R_{\Ss_\pp} \un{\ggamma}_\pp\right) \widetilde{A}_{\Ss_\pp}\right]^{-1}, \label{eqn:Sigma} \end{equation}
		with the matrices $R_{\Ss_\pp}$, $A_{\Ss_\pp}$ and $\widetilde{A}_{\Ss_\pp} = R_{\Ss_\pp}  A_{\Ss_\pp}$ defined in Section \ref{subsec:param}.
	\end{itemize}
\end{proposition}
\begin{proof} See Appendix \ref{app:proofs} in Supplementary Material. \end{proof}

\noindent We may recognise in (\ref{eqn:Sigma}) some sort of `sandwich' covariance matrix expression for a parametric estimator in a misspecified model \citep{Huber67, White82,Kauermann01}. 

\ppn We are now in position to state our main theorem. Recall that $\widehat{\un{\bm p}}_n \doteq \vect(\widehat{\pp}_n)$ and $\widehat{\un{\ggamma}}_n \doteq \vect(\widehat{\ggamma}_n)$, where $\widehat{\pp}_n$ and $\widehat{\ggamma}_n$ are defined in (\ref{eqn:hatpn}) and (\ref{eqn:empcoparray}). Write 
\[\widehat{\un{\ggamma}}_n =  \un{\ggamma}(\widehat{\ttheta}_n) =  \un{\ggamma}^{(q,\Ss_\pp)} + A_{\Ss_\pp} \widehat{\ttheta}_n \]
by plugging (\ref{eqn:hattheta}) into (\ref{eq:un:G:vartheta}). Then, it directly follows from Proposition \ref{prop:thetahat} that:

\begin{theorem} \label{thm:gammahat} Let $(\XX_1,\dots,\XX_n)$ be a sample drawn from $P$, where the support $\Ss_\pp$ of the probability array $\pp$ of $P$ is such that $\Gamma(\Ss_\pp)\neq \emptyset$ (Condition \ref{cond:Gamma:S}). 
	\begin{itemize} 
		\item[(a)] If estimator (\ref{eqn:hatpn}) is a (strongly) consistent estimator of $\pp$; that is, $\widehat{\pp}_n  \overset{\text{P}}{\to} \pp$ ($\widehat{\pp}_n  \overset{\text{as}}{\to} \pp$) as $n \to \infty$, then the empirical copula array $\widehat{\ggamma}_n$  (\ref{eqn:hatgamma}) is a (strongly) consistent estimator of the copula array $\ggamma_\pp$ of $\pp$; that is,
		\begin{equation}  \widehat{\ggamma}_n \overset{\text{P}}{\to} \ggamma_\pp \quad (\widehat{\ggamma}_n \overset{\text{as}}{\to} \ggamma_\pp)\qquad \text{ as } n \to \infty. \label{eqn:consist} \end{equation}
	\end{itemize}
	Assume that Condition \ref{cond:Gamma:S:infinite} holds for $\Ss_\pp$ and that a basis $\Delta_1, \dots, \Delta_{d_\circ}$ of $\bm \Delta^\circ(\Ss_\pp)$ (\ref{eq:DDelta}) has been identified and formed into the matrix $A_{\Ss_\pp}$ as described in Section \ref{subsec:param}. 
	\begin{itemize}
		\item[(b)] If $s_n\left(\un{\widehat{\bm p}}_n- \un{\bm p}\right) \toL \bm Z$ as $n \to \infty$, for some  $s_n \to \infty$ and a given $(\prod_{\ell = 1}^d r_\ell)$-dimensional random vector $\bm Z$, then:
		\begin{equation*}  s_n\left(\un{\widehat{\ggamma}}_n- \un{\ggamma}_\pp\right) \toL 	J_{\ggamma^*}(\un{\bm p}) \bm Z, \label{eqn:gammaasymptZ}\end{equation*}
		where $J_{\ggamma^*}(\un{\bm p})$ is the Jacobian matrix defined in (\ref{eq:Jacobian:gamma:*}) with $\Ss = \Ss_\pp$;
		\item[(c)] In particular, if $(\XX_1,\dots,\XX_n)$ is an i.i.d.\ sample drawn from $P$ so that
		\begin{equation*} \sqrt{n}\,\left(\un{\widehat{\bm p}}_n- \un{\bm p}\right) \toL \Ns_{\prod_{\ell = 1}^d r_\ell}\left(\bm 0,\diag(\un{\bm p}) - \un{\bm p} \, \transp{\un{\bm p}}\right), \label{eqn:classicnorm}  \end{equation*}
		then
		\begin{equation*} \sqrt{n}\,\left(\widehat{\un{\ggamma}}_n - \un{\ggamma}_\pp \right) \toL \Ns_{\prod_{\ell = 1}^d r_\ell}\left(\bm 0, \bm\Sigma_\ggamma \right) \qquad \text{as } n \to \infty, \label{eqn:gammaasympt} \end{equation*}
		where $\bm\Sigma_\ggamma = A_{\Ss_\pp} \bm\Sigma_\ttheta \transp{A}_{\Ss_\pp} $ with $\bm \Sigma_\ttheta$ defined in (\ref{eqn:Sigma}); that is,
		\begin{multline} \bm\Sigma_\ggamma = A_{\Ss_\pp} \bm \left[\transp{\widetilde{A}_{\Ss_\pp}} \diag^{-1}\left( R_{\Ss_\pp} \un{\ggamma}_\pp\right) \widetilde{A}_{\Ss_\pp}\right]^{-1} \left[\transp{\widetilde{A}_{\Ss_\pp}} \diag^{-1}\left( R_{\Ss_\pp} \un{\bm p}\right) \widetilde{A}_{\Ss_\pp} \right] \\ \cdot \left[\transp{\widetilde{A}_{\Ss_\pp}} \diag^{-1}\left( R_{\Ss_\pp} \un{\ggamma}_\pp\right) \widetilde{A}_{\Ss_\pp}\right]^{-1} \transp{A}_{\Ss_\pp}, \label{eqn:Sigmagamma} \end{multline}
		with the matrices $R_{\Ss_\pp}$ and $\widetilde{A}_{\Ss_\pp} = R_{\Ss_\pp}  A_{\Ss_\pp}$  defined in Section \ref{subsec:param}.
	\end{itemize}
\end{theorem}

\noindent Note that $(a)$ holds without Condition \ref{cond:Gamma:S:infinite} for $\Ss_\pp$; in fact, if Condition \ref{cond:Gamma:S} holds but $d_\circ = 0$, the claim is trivial: in such a case $\Gamma(\Ss_\pp)$ contains only $\ggamma^{(q,\Ss_\pp)}$, and thus $\widehat{\ggamma}_n = \ggamma^{(q,\Ss_\pp)} =\ggamma_\pp$ for all $n \geq 1$. Also, in the spirit of Remark \ref{rem:invbasis}, $\bm\Sigma_\ggamma$ is invariant under a change of basis for $\DDelta^\circ(\Ss_\pp)$.

\begin{remark}
	\label{rem:non-iid}
	Statement (b) in both Proposition~4.1 and Theorem~4.1 is formulated in terms of the asymptotic behaviour of the empirical proportion vector $\widehat{\underline{\bm p}}_n$ and therefore applies beyond i.i.d.\ sampling. It requires that $\widehat{\underline{\bm p}}_n$ converges, after normalisation by a sequence $s_n$, to a non-degenerate limiting distribution ${\bm Z}$. In the i.i.d.\ case, the classical central limit theorem gives $s_n=\sqrt{n}$ and a multinomial covariance structure; this is part (c). More generally, for strictly stationary sequences under weak dependence, the empirical proportions remain root-$n$ consistent and asymptotically normal, but with a long-run covariance matrix; see, for example, \cite{Billingsley61} for ergodic Markov chains. Other sampling schemes may lead to different rates $s_n$, as with long-range dependence, where non-Gaussian limits may arise \citep{Taqqu75}. In all such cases, the asymptotic behaviour of the empirical discrete copula process follows from (b) by substituting the appropriate rate and limiting distribution. \qed
\end{remark}

\begin{remark} \label{rem:consistvar} If $\widehat{\pp}_n  \overset{\text{P}}{\to} \pp$ ($\widehat{\pp}_n  \overset{\text{as}}{\to} \pp$) as $n \to \infty$, then $\widehat{\ggamma}_n \overset{\text{P}}{\to} \ggamma_\pp$  ($\widehat{\ggamma}_n \overset{\text{as}}{\to} \ggamma_\pp$) by (\ref{eqn:consist}); and the covariance matrices $\bm \Sigma_\ttheta$ and $\bm \Sigma_\ggamma$ in (\ref{eqn:Sigma}) and (\ref{eqn:Sigmagamma}) may also be (strongly) consistently estimated by 
	\begin{equation}  \widehat{\bm \Sigma}_{\ttheta;n}  = \left[\transp{\widetilde{A}_{\Ss_\pp}} \diag^{-1}\left( R_{\Ss_\pp} \un{\widehat{\ggamma}}_n\right) \widetilde{A}_{\Ss_\pp}\right]^{-1} \left[\transp{\widetilde{A}_{\Ss_\pp}} \diag^{-1}\left( R_{\Ss_\pp} \widehat{\un{\bm p}}_n\right) \widetilde{A}_{\Ss_\pp} \right]\left[\transp{\widetilde{A}_{\Ss_\pp}} \diag^{-1}\left( R_{\Ss_\pp} \un{\widehat{\ggamma}}_n\right) \widetilde{A}_{\Ss_\pp}\right]^{-1}  \label{eqn:hatsigmatheta0} \end{equation}
	\begin{multline} 
		\widehat{\bm \Sigma}_{\ggamma;n} = A_{\Ss_\pp} \left[\transp{\widetilde{A}_{\Ss_\pp}} \diag^{-1}\left( R_{\Ss_\pp} \un{\widehat{\ggamma}}_n\right) \widetilde{A}_{\Ss_\pp}\right]^{-1} \left[\transp{\widetilde{A}_{\Ss_\pp}} \diag^{-1}\left( R_{\Ss_\pp} \widehat{\un{\bm p}}_n\right) \widetilde{A}_{\Ss_\pp} \right] \\ \cdot  \left[\transp{\widetilde{A}_{\Ss_\pp}} \diag^{-1}\left( R_{\Ss_\pp} \un{\widehat{\ggamma}}_n\right) \widetilde{A}_{\Ss_\pp}\right]^{-1} \transp{A}_{\Ss_\pp}. \label{eqn:hatsigmatheta} \end{multline} \qed
\end{remark} 

\noindent In the case  $\Ss_\pp = \Rs$, the expression of $\bm \Sigma_\ggamma$ simplifies slightly, due to $R_\Rs$ being the identity matrix:

\begin{corollary} Let $(\XX_1,\dots,\XX_n)$ be a sample drawn from $P$, where the probability array $\pp$ of $P$ has hyper-rectangular support $\Rs$. Then, if $\widehat{\un{\bm p}}_n$ satisfies (\ref{eqn:empprop}), we have:
	\begin{equation*} \sqrt{n}\,\left(\widehat{\un{\ggamma}}_n - \un{\ggamma}_\pp \right) \toL \Ns_{\prod_{\ell = 1}^d r_\ell}\left(\bm 0,  \bm \Sigma_\ggamma \right) \qquad \text{as } n \to \infty, \end{equation*}
	where $\bm \Sigma_\ggamma = A_\Rs \left[\transp{{A}_\Rs} \diag^{-1}\left( \un{\ggamma}_\pp\right) {A}_\Rs\right]^{-1} \left[\transp{{A}_\Rs} \diag^{-1}\left( \un{\bm p}\right) {A}_\Rs \right]\left[\transp{{A}_\Rs} \diag^{-1}\left( \un{\ggamma}_\pp\right) {A}_\Rs\right]^{-1} \transp{A}_\Rs$.
\end{corollary}

\noindent A similar covariance matrix expression can be found in \cite{Freeman76} and is used in \cite{Little91} in the bivariate case. 

\section{Applications} \label{sec:ill}

\subsection{Yule's coefficient as discrete Spearman's $\rho_\text{S}$, and its empirical estimator} \label{subsec:Yule}

In a bivariate vector $(X_1,X_2)$, the concept of concordance captures the tendency of the two variables to vary in the same direction. \cite{Scarsini84} formalised that concordance is reflected by a partial ordering of the distributions in a given Fr\'echet class. As a result, {\it in continuous vectors}, concordance is intrinsically a property of the underlying (continuous) copula $C_{X_1X_2}$ of the vector $(X_1,X_2)$, and it must be the case that any valid index quantifying concordance is a parameter of $C_{X_1X_2}$. A prominent example is Spearman’s $\rho_{\mathrm S}$ \citep[Section 5.1.2]{Nelsen06}, defined as the (Pearson's) correlation of $C_{X_1X_2}$. 

\ppn However, these ideas cannot be transferred directly to discrete random vectors, as the classical copula is ill-defined. Indeed \citet[Section 3]{Scarsini84} admitted that the above formalisation of concordance must be replaced by a `{\it (necessarily) more complicated one}' (p.\,211) and proposed a variant based on the subcopula (\citealp[Definition 3]{Schweizer74}), which he judged himself `{\it cumbersome}' and `{\it not easily testable in practice}' (p.\,213). It is also well-documented that naive extensions of Spearman’s $\rho_{\mathrm S}$ to discrete data suffer from a variety of deficiencies \citep[Sections 4.2-4.4]{Genest07}. 

\ppn Now, equipped with discrete copulas, an unambiguous notion of concordance emerges seamlessly. In particular, inspired by \cite{Tchen80}, we define in perfect analogy with \citet[Definition 4]{Scarsini84}: 
\begin{definition} Let $(X_1,X_2)$ and $(X'_1, X'_2)$ be two bivariate discrete vectors with distributions $P$ and $P'$ and associated probability arrays $\pp,\pp' \in \R^{r_1 \times r_2}$. Assume that $P$ and $P'$ admit respective copula arrays $\ggamma_\pp \doteq \left(\gamma_{i_1i_2}\right)_{i_1 \in [r_1],i_2\in [r_2]}$ and $\ggamma_{\pp'} \doteq \left(\gamma'_{i_1i_2}\right)_{i_1 \in [r_1],i_2\in [r_2]}$ and copula measures $G_\pp$ and $G_{\pp'}$. Then, we say that  $(X_1,X_2)$ shows more concordance than $(X'_1, X'_2)$ if $G_\pp$ stochastically dominates $G_{\pp'}$; i.e.,
	\[ \forall (j_1,j_2) \in [r_1] \times [r_2],\quad \sum_{i_1=1}^{j_1} \sum_{i_2=1}^{j_2} \gamma_{i_1i_2} \geq \sum_{i_1=1}^{j_1} \sum_{i_2=1}^{j_2}  \gamma'_{i_1i_2}. \]
\end{definition}
\noindent Accordingly, any index quantifying the concordance within a discrete vector $(X_1,X_2)$ must be a parameter of its discrete copula. \cite{Geenens2020} proposed a discrete analogue to Spearman's $\rho_\text{S}$ called `Yule's coefficient $\Upsilon$', and defined as the (Pearson's) correlation coefficient associated to the {discrete copula measure} under consideration. Recall that $G_\pp$ is supported on $\{\frac{1}{r_1+1},\ldots,\frac{r_1}{r_1+1}\} \times \{\frac{1}{r_2+1},\ldots,\frac{r_2}{r_2+1}\}$. Then, for a distribution admitting the copula array $\ggamma_\pp \doteq \left(\gamma_{i_1i_2}\right)_{i_1 \in [r_1],i_2\in [r_2]}$, Yule's coefficient is defined as 
\begin{equation} \Upsilon = \Upsilon (\ggamma_\pp) \doteq   \frac{12}{\sqrt{(r_1+1)(r_1-1)(r_2+1)(r_2-1)}}\sum_{i_1 \in [r_1]}\sum_{i_2 \in [r_2]} i_1 i_2 {\gamma}_{i_1i_2} - 3 \frac{\sqrt{(r_1+1)(r_2+1)}}{\sqrt{(r_1-1)(r_2-1)}}.  \label{eqn:Yule} \end{equation}
Call the vectors $(r_1) \doteq \transp{(1 \ 2 \ \ldots \ r_1)}$ and $(r_2) \doteq \transp{(1 \ 2 \ \ldots \ r_2)}$, and see that 
\begin{equation*} \sum_{i_1 \in [r_1]}\sum_{i_2 \in [r_2]} i_1 i_2 \gamma_{i_1i_2} =  \transp{(r_1)} \ggamma_\pp (r_2) =  \transp{[(r_2) \otimes (r_1)]} \un{\ggamma}_\pp, \end{equation*} 
where $(r_2) \otimes (r_1)$ is the Kronecker product between $(r_2)$ and $(r_1)$. Evidently, $\Upsilon$ is a linear functional of $\ggamma_\pp$ (or $\un{\ggamma}_\pp$), and it is straightforward to verify that it satisfies the (discrete analogues of the) characterising properties ($\mu 1$)-($\mu 6$) of valid concordance coefficients listed in \citet[Definition 2.4.7]{Durante15}. (The property ($\mu 7$) will follow from the developments below.) 
As noted in \cite{Geenens2020},  $\Upsilon =1$ if and only if $\ggamma $ is supported on the first diagonal, and $\Upsilon =-1$ if and only if  $\ggamma $ is supported on the anti-diagonal; in both cases, $\ggamma$ describes a monotonic one-to-one relationship between $X_1$ and $X_2$.

\ppn Now, consider the empirical framework of Section \ref{sec:empcop}, in which we have $n$ copies $(\XX_1,\dots,\XX_n)$ of $\XX = (X_1,X_2) \sim P$. From (\ref{eqn:hatpn}), (\ref{eqn:empcoparray}) and (\ref{eqn:Yule}), the natural `plug-in' estimator of Yule's coefficient $\Upsilon$ is
\begin{equation} \widehat{\Upsilon}_n = \Upsilon(\widehat{\ggamma}_n)  =  \frac{12}{\sqrt{(r_1+1)(r_1-1)(r_2+1)(r_2-1)}}\sum_{i_1 \in [r_1]}\sum_{i_2 \in [r_2]} i_1 i_2 \widehat{\gamma}_{n;i_1i_2} - 3 \frac{\sqrt{(r_1+1)(r_2+1)}}{\sqrt{(r_1-1)(r_2-1)}}\label{eqn:Upshat}
\end{equation}  
which will be referred to as the {\it empirical Yule's coefficient}. Let $\kappa_{r_1,r_2} \doteq \frac{12}{\sqrt{(r_1+1)(r_1-1)(r_2+1)(r_2-1)}}$. Then $\widehat{\Upsilon}_n - \Upsilon =  \kappa_{r_1,r_2} \transp{[(r_2) \otimes (r_1)]} (\widehat{\un{\ggamma}}_n - \un{\ggamma}_\pp)$, and the asymptotic distribution of $\widehat{\Upsilon}_n$ follows as a straightforward consequence of Theorem \ref{thm:gammahat}. 

\begin{theorem} \label{thm:upsilonhat} Let $(\XX_1,\dots,\XX_n)$ be a sample drawn from $P$, where the support $\Ss_\pp$ of the probability array $\pp$ of $P$ is such that $\Gamma(\Ss_\pp)\neq \emptyset$ (Condition \ref{cond:Gamma:S}).  
	\begin{itemize} 
		\item[(a)] If the estimator $\widehat{\pp}_n$ in (\ref{eqn:hatpn}) is a (strongly) consistent estimator of $\pp$; that is, $\widehat{\pp}_n  \overset{\text{P}}{\to} \pp$ ($\widehat{\pp}_n  \overset{\text{as}}{\to} \pp$) as $n \to \infty$, then the empirical Yule's coefficient $\widehat{\Upsilon}_n$  (\ref{eqn:Upshat}) is a (strongly) consistent estimator of Yule's coefficient $\Upsilon$; that is,
		\[ \widehat{\Upsilon}_n \overset{\text{P}}{\to} \Upsilon \quad (\widehat{\Upsilon}_n \overset{\text{as}}{\to} \Upsilon)\qquad \text{ as } n \to \infty.\]
	\end{itemize}
	Assume Condition \ref{cond:Gamma:S:infinite} holds for $\Ss_\pp$ and a basis $\Delta_1, \dots, \Delta_{d_\circ}$ of $\bm \Delta^\circ(\Ss_\pp)$ in (\ref{eq:DDelta}) has been identified and formed into the matrix $A_{\Ss_\pp}$ as described in Section \ref{subsec:param}. 
	
	\begin{itemize} 
		\item[(b)] If $s_n\left(\un{\widehat{\bm p}}_n- \un{\bm p}\right) \toL \bm Z$ as $n \to \infty$, for some  $s_n \to \infty$ and a given $(\prod_{\ell = 1}^d r_\ell)$-dimensional random vector $\bm Z$; then:
		\begin{equation*}  s_n\left({\widehat{\Upsilon}}_n- {\Upsilon}\right) \toL \kappa_{r_1,r_2} \transp{[(r_2) \otimes (r_1)]}	J_{\ggamma^*}(\un{\bm p}) \bm Z, \label{eqn:upsasymptZ}\end{equation*}
		where $J_{\ggamma^*}(\un{\bm p})$ is the Jacobian matrix defined in (\ref{eq:Jacobian:gamma:*}) with $\Ss = \Ss_\pp$;
		\item[(c)] In particular, if $(\XX_1,\dots,\XX_n)$ is an i.i.d.\ sample drawn from $P$ so that
		\begin{equation*} \sqrt{n}\,\left(\un{\widehat{\bm p}}_n- \un{\bm p}\right) \toL \Ns_{\prod_{\ell = 1}^d r_\ell}\left(\bm 0,\diag(\un{\bm p}) - \un{\bm p} \, \transp{\un{\bm p}}\right),   \end{equation*}
		then
		\begin{equation} \sqrt{n}\,\left(\widehat{{\Upsilon}}_n - {\Upsilon} \right) \toL \Ns\left(\bm 0, \sigma^2_\Upsilon \right) \qquad \text{as } n \to \infty, \label{eqn:upsasympt} \end{equation}
		where  $\sigma^2_\Upsilon = \kappa^2_{r_1,r_2} \transp{[(r_2) \otimes (r_1)]}\bm\Sigma_\ggamma  {[(r_2) \otimes (r_1)]}$ and $\bm \Sigma_\ggamma$ is defined in (\ref{eqn:Sigmagamma}).  
	\end{itemize}
\end{theorem}
\noindent From Remark \ref{rem:consistvar}, the asymptotic variance of $\widehat{\Upsilon}_n$ can be (strongly) consistently estimated by 
\begin{equation} \widehat{\sigma}^2_{\Upsilon;n} = \kappa^2_{r_1,r_2} \transp{[(r_2) \otimes (r_1)]}   \widehat{\bm \Sigma}_{\ggamma;n} {[(r_2) \otimes (r_1)]}. \label{eqn:varUpsest} \end{equation}

\begin{example} Table \ref{tab:happiness} displays a cross‐classification of $n=2,955$ respondents from the 2006 US General Social Survey reproduced from \citet[Table 2.3]{Agresti12}. Columns correspond to self‐reported happiness levels (`Very happy', `Pretty happy', `Not too happy'), and rows to self‐assessed family income relative to the national average (`Above average', `Average', `Below average'). In an attempt to answer the question whether or not money can buy happiness, we may want to highlight some concordance (i.e., positive association) between happiness levels and family income.
	
	\begin{table}[ht]
		\centering
		\begin{tabular}{lccc c}
			\toprule
			& \multicolumn{3}{c}{{\bf Happiness }}   &      \\
			\cmidrule(lr){2-4}
			{\bf Family Income}    & Very happy & Pretty happy  & Not too happy & Total \\
			\midrule
			Above average         & 272           & 294     & 49           & 615   \\
			Average    & 454           & 835     & 131           & 1420  \\
			Below average    & 185            & 527     & 208           & 920   \\
			\midrule
			Total  & 911           & 1656    & 388           & 2955  \\
			\bottomrule
		\end{tabular}
		\caption{Happiness versus Family Income}
		\label{tab:happiness}
	\end{table}
	
	\ppn For these data, the estimators (\ref{eqn:hatpn}) and (\ref{eqn:empcoparray}) return 
	\begin{equation} \widehat{\pp}_n = \begin{pmatrix} 0.0921 & 0.0995 & 0.0166 \\ 
			0.1536 & 0.2825 & 0.0444 \\
			0.0626 & 0.1783 & 0.0704
		\end{pmatrix} \quad \text{ and } \quad \widehat{\ggamma}_n = \begin{pmatrix} 0.1574 & 0.1024 & 0.0735 \\ 
			0.1167 & 0.1293 & 0.0873 \\
			0.0592 & 0.1016 & 0.1725 \end{pmatrix},  \label{eqn:empex} \end{equation}  
	the (smoothed) empirical probability array and the corresponding empirical copula array (obtained from $\widehat{\pp}_n$ by IPF; see Remark \ref{rmk:IPF}). Figure \ref{fig:Ex51} shows bubble plots of these two probability arrays. The unbalanced margins obscure the association between happiness and income in the original array $\widehat{\pp}_n$ (left), whereas the copula array (right) reveals concordance through the concentration of probability along the main diagonal. Computing the empirical Yule's coefficient from $\widehat{\ggamma}_n$ in~(\ref{eqn:empex}) using~(\ref{eqn:Upshat}) gives $\widehat{\Upsilon}_n=0.2956$, indicating positive concordance between happiness and family income. For inference, we assume respondents were randomly drawn from a target population and answered independently, allowing us to apply Theorem \ref{thm:upsilonhat}$(c)$. Since Table \ref{tab:happiness} has no null cells, the underlying distribution has rectangular support. We can therefore use the canonical basis~(\ref{eqn:canon33}) in~(\ref{eqn:hatsigmatheta}) and~(\ref{eqn:varUpsest}), with $R_{\Ss_\pp}$ the $9\times9$ identity matrix. Taking $\kappa_{3,3}=1.5$, (\ref{eqn:varUpsest}) gives $\widehat{\sigma}^2_{\Upsilon;n}=2.1121$ and a standard error of $\widehat{\sigma}_{\Upsilon;n}/\sqrt n=0.0267$. The resulting approximate Wald 95\%-confidence interval from~(\ref{eqn:upsasympt}) is $[0.2432,\,0.3480]$. Thus, the positive association is highly significant but of low magnitude. \qed
	
	\begin{figure}
		\centering
		\includegraphics[width=0.66\textwidth]{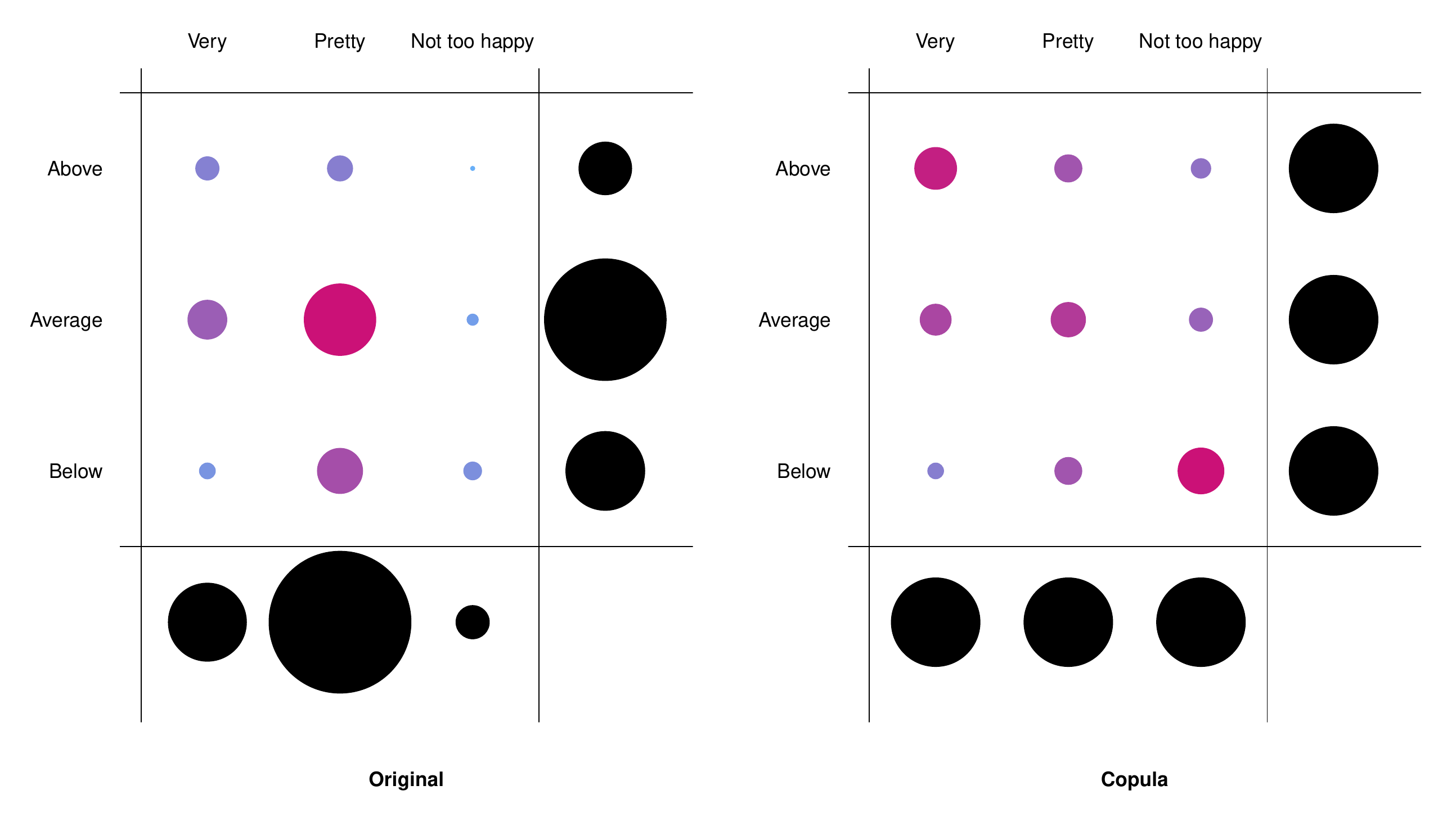}
		\caption{Bubble plots for the original empirical probability array $\widehat{\pp}_n$ (left) and empirical copula array $ \widehat{\ggamma}_n$ (right) for the data shown in Table \ref{tab:happiness}. The original category labels are shown for ease of interpretation but could be omitted.} 
		\label{fig:Ex51}
	\end{figure} 
	
	\ppn This confirms that the specific values $x^{(1)}_1,\ldots,x^{(1)}_{r_1}$ and $x^{(2)}_1,\ldots,x^{(2)}_{r_2}$ defining the initial domains $\Xs_1$ and $\Xs_2$ (Section \ref{sec:bivdiscrcop}) play no role in the discrete copula machinery. They need not be known or even have any physical reality, provided $\Xs_1$ and $\Xs_2$ are ordered sets, as in Table \ref{tab:happiness}. In particular, estimating $\Upsilon$ and conducting the corresponding inference require no arbitrary scores for the table’s categories, unlike a classical analysis \citep[Section 2.1]{Agresti12}. This is a substantial benefit: \citet[p.\,740]{Goodman54} already observed that assigning such scores was `{\it infrequently appropriate}'.
	
\end{example}

\begin{remark} 	Yule’s coefficient is the analogue of Spearman's $\rho_S$ for discrete distributions. Other classical concordance indices, such as Kendall’s $\tau$, can be adapted analogously to discrete settings as functionals of the discrete copula array $\ggamma_\pp$; see \cite{Kojadinovic24}. The asymptotic behaviour of empirical estimators would follow similarly from Theorem~4.1 via the delta method. \qed
\end{remark}

\subsection{Testing for quasi-independence} \label{subsec:quasitest}

Quasi-independence refers to random variables that are {\it independent within their joint support} \citep{Goodman68}. For a $d$-dimensional discrete random vector $\mathbf{X} = (X_1,\ldots,X_d)$ with distribution $P$ and probability array $\mathbf{p}$ on support $\mathcal{S}_\mathbf{p}$, quasi-independence translates into:
\[p_\ii  = \left\{\prod_{\ell =1}^d \beta^{(\ell)}_{i_\ell}\right\} \indic{\bm i \in \Ss_\pp},\qquad \forall \ii = (i_1,\ldots,i_d) \in \Rs,\]
for $\{\beta^{(\ell)}_{i_\ell}\}_{\ell \in [d], i_\ell \in [r_\ell]}$ a set of positive constants. When $\mathcal{S}_\mathbf{p} = \mathcal{R}$ (hyper-rectangular support), this is equivalent to the joint independence of the variables $(X_1,\ldots,X_d)$. For a non-hyper-rectangular support, the essence of quasi-independence is that the variables behave as if they were independent {\it except} that certain combinations of values cannot occur together (i.e., the structural zeros of $\pp$). If $\Gamma(\Ss_\pp) \neq \emptyset$ (Condition \ref{cond:Gamma:S}), the copula array of a distribution showing quasi-independence on $\Ss_\pp$ is naturally the quasi-independence copula array $\ggamma^{(q,\Ss_\pp)}$ introduced in Section \ref{sec:aff}.

\ppn Assume that $\Ss_\pp$ satisfies Condition \ref{cond:Gamma:S:infinite}, and that a basis for $\un{\DDelta}^\circ(\Ss_\pp)$ in (\ref{eq:DDelta:vect}) has been identified and formed into a matrix $A_{\Ss_\pp}$ (Section \ref{subsec:param}). Then Lemma \ref{lem:bij} implies that $\pp$ shows quasi-independence if and only if $\ttheta = \zero$ when its copula array $\ggamma_\pp$ is expressed as in (\ref{eq:gamma:decomp}). Thus, testing for quasi-independence within $\pp$ amounts to testing for the null hypothesis $H_0 : \ttheta = \zero$ based on a sample $(\XX_1,\dots,\XX_n) \sim P$. 

\ppn From the empirical probability array $\widehat{\pp}_n$ in (\ref{eqn:hatpn}), we can compute $\widehat{\ttheta}_n$ as in (\ref{eqn:hattheta}) where $\ggamma(\ttheta) = \ggamma^{(q,\Ss_\pp)} + \vect^{-1}(A_{\Ss_\pp}\ttheta)$. (Alternatively, from the empirical copula array $\widehat{\ggamma}_n$, we may identify $\widehat{\ttheta}_n$ by solving $A_{\Ss_\pp}\widehat{\ttheta}_n = \un{\widehat{\ggamma}}_n- \un{\ggamma}^{(q,\Ss_\pp)}$.) Assume that $(\XX_1,\dots,\XX_n)$ is an i.i.d.\ sample. Then, Proposition \ref{prop:thetahat}$(c)$ establishes that, under $H_0$, $\sqrt{n} \, \widehat{\ttheta}_n \toL \Ns_{d_\circ}(\zero,\bm \Sigma_\ttheta)$ as $n \to \infty$, and therefore $n \, \transp{\widehat{\ttheta}}_n \bm \Sigma^{-1}_\ttheta \widehat{\ttheta}_n \toL \chi^2_{d_\circ}$, noting that $\bm\Sigma_\ttheta$ is of (full) rank $d_\circ$. Estimating $\bm\Sigma_\ttheta$ by $\widehat{\bm\Sigma}_{\ttheta;n}$ in (\ref{eqn:hatsigmatheta0}), we may compute $t = n \, \transp{\widehat{\ttheta}}_n \widehat{\bm\Sigma}^{-1}_{\ttheta;n} \widehat{\ttheta}_n$ for the observed data, which yields an (approximated) $p$-value $p = \P(W > t)$, $W \sim \chi^2_{d_\circ}$. The decision follows.

\begin{example} Table \ref{tab:teens} presents data collected by \cite{Brunswick71} on $n=291$ teenagers' health concerns, gathered by asking: `{\it What are some of things you'd like to talk over with a doctor, if you had the chance to ask him anything you wanted?}' The data is cross-classified by age (12-15 years old versus 16-17 years old), gender, and four categories of health concerns (or absence thereof). No teenage boys reported concerns about menstrual problems. Since this is unlikely to reflect sampling variation, we treat the corresponding cells as structural zeros in $\pp$. The resulting non-hyperrectangular support $\Ss_\pp$ rules out joint independence of gender, age, and health concerns. We therefore ask whether any association remains beyond this structural incompatibility; if not, the variables are quasi-independent.
	
	\begin{table}[htbp]
		\centering
		
		\begin{tabular}{llccc}
			\toprule
			\textbf{Gender} & \textbf{Health concern} & \textbf{12--15 yo} & \textbf{16--17 yo} & \textbf{Total} \\
			\midrule
			\multirow{4}{*}{Males} 
			& Sex, reproduction    & 4   & 2   & 6   \\
			& Menstrual problems   & --  & --  & --  \\
			& How healthy am I     & 42  & 7   & 49  \\
			& Nothing              & 57  & 20  & 77  \\
			& \textbf{Total}       & 103 & 29  & 132 \\
			\midrule
			\multirow{4}{*}{Females} 
			& Sex, reproduction    & 9   & 7   & 16  \\
			& Menstrual problems   & 4   & 8   & 12  \\
			& How healthy am I     & 19  & 10  & 29  \\
			& Nothing              & 71  & 31  & 102 \\
			& \textbf{Total}       & 103 & 56  & 159 \\
			\bottomrule
		\end{tabular}
		\vspace{1ex}
		\caption{Teenagers’ concern with health problems based on data from \cite{Brunswick71}.}
		\label{tab:teens}
	\end{table}
	
	\ppn The (smoothed) empirical probability array (\ref{eqn:hatpn}) and empirical copula array (\ref{eqn:empcoparray}) are found to be
	\begin{equation}  \widehat{\pp}_n =
		\left\{
		\begin{bmatrix}
			0.0139 & 0.0071 \\
			0 & 0 \\
			0.1441 & 0.0242 \\
			0.1955 & 0.0687
		\end{bmatrix}, 
		\begin{bmatrix}
			0.0311 & 0.0242 \\
			0.0139 & 0.0276 \\
			0.0653 & 0.0345 \\
			0.2434 & 0.1064
		\end{bmatrix}
		\right\}, \quad  \widehat{\ggamma}_n =
		\left\{
		\begin{bmatrix}
			0.0662 & 0.0617 \\
			0 & 0 \\
			0.1549 & 0.0473 \\
			0.1038 & 0.0660
		\end{bmatrix}, 
		\begin{bmatrix}
			0.0507 & 0.0714 \\
			0.0550 & 0.1950 \\
			0.0244 & 0.0233 \\
			0.0448 & 0.0354
		\end{bmatrix}
		\right\},
		\label{eqn:teenhat}
	\end{equation}
	where $\widehat{\ggamma}_n$ is the $\Is$-projection of $\widehat{\pp}_n$ onto the copula Fr\'echet class (\ref{eqn:Gamma422}), readily obtained via a trivariate version of the IPF procedure (Remark \ref{rmk:IPF}). Recognising that the support $\Ss_\pp$ of $\pp$ is the one studied in Example \ref{ex:example2}, we know that $d_\circ = 8$. Working with the basis matrix $A_{\Ss_\pp}$ shown in Appendix \ref{app:ex2}, we get
	\[\widehat{\ttheta}_n = \transp{(0.0217; -0.1164; -0.0234;  0.0126;  0.0556;  0.0376; -0.0115;  0.0162)}, \]
	and with $\widehat{\bm\Sigma}_{\ttheta;n}$ as in (\ref{eqn:hatsigmatheta0}), we compute $ t = 31.49$. This yields a $p$-value of $p = \P( W > 31.49) = 0.0001$ for $W \sim \chi^2_8$; hence a clear rejection of the hypothesis of quasi-independence at all reasonable levels. 
	
	\ppn Figure~\ref{fig:Ex52} compares the empirical copula array $\widehat{\ggamma}_n$ in~(\ref{eqn:teenhat}) (two slices, left) with the quasi-independence copula array in~(\ref{eqn:coparrayqSp}) (two slices, right), and helps explain the rejection: females more often raise questions about sex/reproduction, particularly menstrual problems among those aged 16–17, than under quasi-independence. Males aged 12–15 are more concentrated in general health and no-concern categories. Gender and age group are also associated, although this likely reflects the sampling design.     \qed
	
	\begin{figure}
		\centering
		\includegraphics[width=0.65\textwidth]{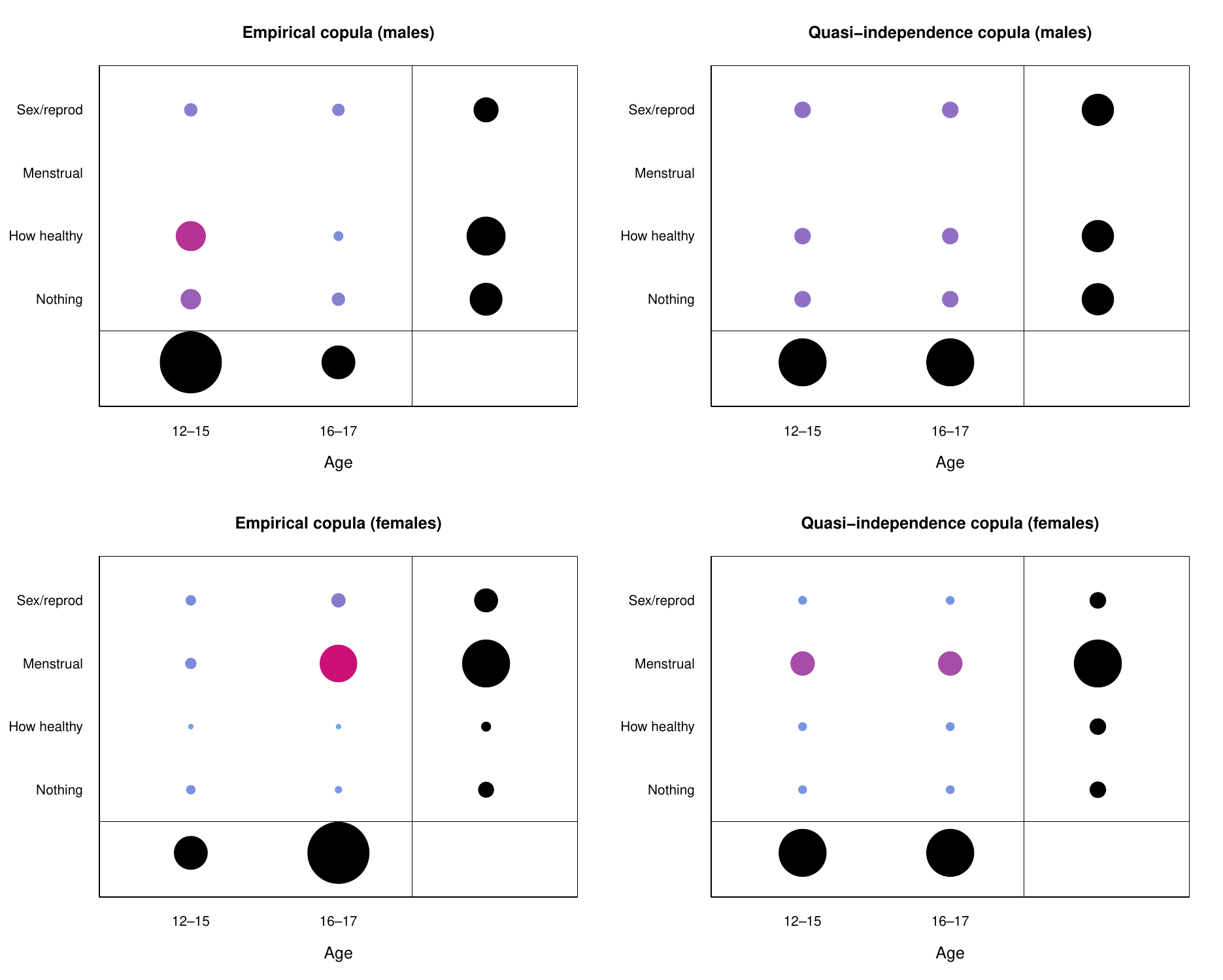}
		\caption{Bubble plots of the empirical discrete copula array $\widehat{\ggamma}_n$ (left) and the corresponding quasi-independence copula array $\ggamma^{(q,\Ss_\pp)}$ (right) for the data shown in Table \ref{tab:teens}. The original category labels are shown for ease of interpretation but could be omitted.} 
		\label{fig:Ex52}
	\end{figure} 
	
\end{example}

\section{Conclusion} \label{sec:ccl}

The discrete copula proposed in \cite{Geenens2020} opened new perspectives for applying copula-like methods to discrete distributions. The present paper provides full theoretical justification  for that construction through Proposition \ref{prop:sklar:like}, which may be regarded as a discrete analogue of Sklar's theorem. Articulated around $\Is$-projections \citep{Csiszar75} on Fr\'echet classes of distributions, it captures the very essence of copula models by establishing a clear separation between marginals and dependence -- unlike classical Sklar's formalism when applied to discrete distributions \citep{Geenens23short}. The paper develops a complete inferential framework for empirical estimation of the discrete copula in arbitrary dimension.  Under a sensible support condition, the empirical discrete copula estimator is defined as the $\Is$-projection of the empirical probability array onto the relevant uniform-margins Fr\'echet class. This projection is straightforward to compute in practice via the Iterative Proportional Fitting procedure.  

\ppn Our main theoretical result, Theorem~\ref{thm:gammahat}, can be viewed as a conceptual analogue, in the discrete setting, of the weak convergence results available for the empirical copula process in the continuous case \citep{Fermanian04,Segers12}. While the technical result differs substantially between the continuous and discrete frameworks, its role is similar: Theorem~\ref{thm:gammahat} provides the basic asymptotic foundation on which large-sample inference for discrete copula models can be built. In particular, it establishes that the empirical discrete copula array is $\sqrt{n}$-consistent and asymptotically normal under standard random sampling, while accommodating more general settings as well.


\ppn Ongoing and future work includes the case of countably infinite supports, which presents its particular challenges, and inference for discrete parametric copula models in higher dimensions, extending bivariate ideas exposed in \citet[Section 3.2]{Kojadinovic24}. In addition, given the perfect analogy between the definition (\ref{eqn:defcopKL}) of the discrete copula and the definition of the entropy-regularised optimal transport plan between discrete distributions \citep[equation (4.7)]{Peyre19}, Theorem \ref{thm:gammahat} may also be relevant to empirical optimal transport problems.  

\bibliographystyle{elsarticle-harv-cond-mod}
\setlength{\bibsep}{0cm}
\def\bibfont{\footnotesize} 
\bibliography{libraries-copula}

\appendix 

	\section{Proofs} \label{app:proofs}
	
	\subsection*{Proof of Lemma \ref{lem:frechet}}
	
	The Fr\'echet class $\Fs(\bm q^{(1)}, \dots, \bm q^{(d)})$ is non-empty, as it contains at least the independence array (\ref{eqn:indeparray}).  
	
	\ppn For $\pp, \pp' \in  \Fs(\bm q^{(1)}, \dots, \bm q^{(d)})$ and $w \in [0,1]$, $[w \pp + (1-w) \pp']^{(\ell)} = w \bm p^{(\ell)} + (1-w) \bm p'^{(\ell)} =w \bm q^{(\ell)} + (1-w) \bm q^{(\ell)} =  \bm q^{(\ell)}$ $\forall \ell \in [d]$, so  $w \pp + (1-w) \pp' \in \Fs(\bm q^{(1)}, \dots, \bm q^{(d)})$. Thus $\Fs(\bm q^{(1)}, \dots, \bm q^{(d)})$ is convex.
	
	\ppn The total variation distance between two probability measures $P$ and $P'$ on $\Xs$ is half the $\ell_1$-distance between their probability arrays $\pp$ and $\pp'$. Let $\{\pp_n\}_{n \geq 1} \in \Fs(\bm q^{(1)}, \dots, \bm q^{(d)})$ be a sequence of probability arrays converging to some probability array $\pp$. A necessary condition for this is that $\bm p_n^{(\ell)}$ converges to $\bm p^{(\ell)}$ $\forall \ell \in [d]$; and as $\bm p_n^{(\ell)}  = \bm q^{(\ell)}$ $\forall n \geq 1$, $\bm p^{(\ell)} = \bm q^{(\ell)}$ and $\pp \in \Fs(\bm q^{(1)}, \dots, \bm q^{(d)})$. Thus $\Fs(\bm q^{(1)}, \dots, \bm q^{(d)})$ is closed in total variation. \qed

	\subsection*{Proof of Proposition~\ref{prop:sklar:like}}
	
	$(i) \Rightarrow (ii)$ has been stated earlier, and follows directly from Corollary \ref{cor:suppproj}.
	
	\ppn $(ii)\Rightarrow(iii)$. Assume that $\pp$ admits a unique $\mathcal I$-projection onto $\Gamma$, denoted $\ggamma_\pp=\Is_\Gamma(\pp)$, with $\Supp(\ggamma_\pp)=\Supp(\pp)$. By Lemma \ref{lem:multrepres}, there exist positive constants $\{\beta^{(\ell)}_{i_\ell}\}$ such that
	\[ \gamma_{\pp;\bm i} = p_{\bm i}\prod_{\ell=1}^d \beta^{(\ell)}_{i_\ell}, \qquad \bm i\in\Rs. \]
	Thereby,
	\[ p_{\bm i} = \gamma_{\pp;\bm i} \prod_{\ell=1}^d \frac{1}{\beta^{(\ell)}_{i_\ell}} , \qquad \bm i\in\Rs. \]
	Since $\pp\in\Fs(\bm p^{(1)},\ldots,\bm p^{(d)})$ and $\Supp(\ggamma_\pp)=\Supp(\pp)$, Lemma \ref{lem:multrepres} shows that $\pp$ is indeed the $\mathcal I$-projection of $\ggamma_\pp$ onto
	$\Fs(\bm p^{(1)},\ldots,\bm p^{(d)})$. Uniqueness of both $\mathcal I$-projections
	then implies that $\ggamma_\pp$ is the unique element of $\Gamma$ with this property,
	establishing~(iii).
	
	\ppn $(iii)\Rightarrow(i)$. Assume that there exists a unique $\ggamma\in\Gamma$ such that $\pp=\Is_{\Fs(\bm p^{(1)},\ldots,\bm p^{(d)})}(\ggamma)$. By properties of the $\mathcal I$-projection, one must have $\Supp(\pp)\subseteq\Supp(\ggamma)$ (otherwise $\Is(\pp\|\ggamma)=\infty$ and $\pp$ cannot be the unique $\Is$-projection). Similarly $\Supp(\ggamma)\subseteq\Supp(\pp)$, since $\ggamma$ is the $\mathcal I$-projection of $\pp$ onto $\Gamma$. Hence $\Supp(\ggamma)=\Supp(\pp)=\Ss_{\pp}$, and since $\ggamma\in\Gamma$, this implies $\Gamma(\Ss_{\pp})\neq\emptyset$, that is,~(i).
	\qed

	\subsection*{Proof of Lemma \ref{lem:bij}}

	Assume Conditions \ref{cond:Gamma:S}-\ref{cond:Gamma:S:infinite}, and let $\bm \gamma \in \Gamma(\Ss)$. Then, $\vect(\bm \gamma - \ggamma^{(q,\Ss)}) \in \un{\bm \Delta}^\circ(\Ss)$ and there exists $\bm \theta \in \R^{d_\circ}$ such that $A_\Ss \bm \theta = \vect(\bm \gamma - \ggamma^{(q,\Ss)})$ since $\un{\Delta}_1, \dots, \un{\Delta}_{d_\circ}$ is a basis of $\un{\bm \Delta}^\circ(\Ss)$. In other words, there exists a $\bm \theta \in \R^{d_\circ}$ such that $\ggamma(\bm \theta) = \ggamma$. In fact, $\ttheta$ belongs to $\Theta$, as $\ggamma(\ttheta) \in \Gamma(\Ss)$. Hence $\ggamma: \Theta \to \Gamma(\Ss)$ is surjective.	As the matrix $A_\Ss$ has full column rank (since $\un{\Delta}_1, \dots, \un{\Delta}_{d_\circ}$ are linearly independent) with $d_\circ \leq \prod_{\ell = 1}^d r_\ell$ columns, the function from $\R^{d_\circ}$ to $\R^{\prod_{\ell = 1}^d r_\ell}$ defined by $\bm \theta \mapsto A_\Ss \bm \theta$ is injective, which implies that $\ggamma: \Theta \to \Gamma(\Ss)$ is also injective.  \qed

	\subsection*{Proof of Proposition \ref{prop:diff}}

	As, by definition, all $\pp \in \Ps(\Ss)$ are such that $p_\ii = 0$ for $\ii \in \Rs \backslash \Ss$, the function $\ttheta^*: \un{\Ps}(\Ss) \to \Theta$ is effectively a function of the vector $\bm p_\Ss \doteq \{p_\ii\}_{\ii \in \Ss} \in (0,1)^{|\Ss|}$ only. We make this explicit by defining the function $\widetilde{\ttheta}^*(\bm p_\Ss) \doteq \ttheta^*(\transp{R}_\Ss \bm p_\Ss)$. We can write:
	\begin{align*} 
		\widetilde{\ttheta}^*(\bm p_\Ss) & = \arg \min_{\ttheta \in \Theta} \Is\left(\ggamma(\ttheta) \| \vect^{-1} (\transp{R}_\Ss \bm p_\Ss)\right) = \arg \min_{\ttheta \in \Theta} \Is\left(\vect^{-1}(\un{\ggamma}(\ttheta)) \| \vect^{-1} (\transp{R}_\Ss \bm p_\Ss)\right) \\
		& = \arg \min_{\ttheta \in \Theta} \Is\left( \un{\ggamma}(\ttheta) \| \transp{R}_\Ss \bm p_\Ss\right) = \arg \min_{\ttheta \in \Theta} \Is\left( R_\Ss \un{\ggamma}(\ttheta) \|  \bm p_\Ss\right)
	\end{align*} 
	where, on the second line, we slightly abuse the notation and let (\ref{eq:KL:divergence}) be applicable as-is to vectors.

	\ppn By Lemma \ref{lem:bij}, $\un{\ggamma}(\cdot)$ in (\ref{eq:un:G:vartheta}) is an affine bijection between $\Theta$ and $\un{\Gamma}(\Ss)$. It follows that $R_\Ss \un{\ggamma}(\cdot)$ 
	is an affine bijection from $\Theta$ to $\un{\Gamma}_\Ss(\Ss) \doteq \left\{R_\Ss \un{\ggamma}: \un{\ggamma} \in \un{\Gamma}(\Ss) \right\}$. We can now apply Corollary 31 in \cite{Geenens24-diff}, with $\phi(x) = x\log x$ (which corresponds to the $\Is$-divergence (\ref{eq:KL:divergence})); see that $\Theta$ in (\ref{eq:Theta}) is open and defined by linear constraints, while $\un{\Gamma}_\Ss(\Ss) \subset (0,1)^{|\Ss|} \subset [0,1]^{|\Ss|}$ is a convex subset of probability vectors. We conclude that $\widetilde{\ttheta}^*$ is continuously differentiable, with Jacobian matrix at $\bm p_\Ss$ given by
	\[ J_{\widetilde{\ttheta}^*}(\bm p_\Ss) = \left[\transp{\widetilde{A}_\Ss} \diag^{-1}\left( R_\Ss \un{\ggamma}(\widetilde{\ttheta}^*(\bm p_\Ss))\right) \widetilde{A}_\Ss\right]^{-1} \transp{\widetilde{A}_\Ss} \diag^{-1}\left( \un{\bm p}_\Ss\right). \]
	As $\ttheta^*( \un{\bm p}) = \widetilde{\ttheta}^*(R_\Ss \un{\bm p})$, we get from the chain rule and (\ref{eq:gamma:*}):
	\begin{align*} J_{\ttheta^*}(\un{\bm p}) = J_{\widetilde{\ttheta}^*}(R_\Ss \un{\bm p}) R_\Ss & =   \left[\transp{\widetilde{A}_\Ss} \diag^{-1}\left( R_\Ss \un{\ggamma}(\widetilde{\ttheta}^*(R_\Ss \un{\bm p}))\right) \widetilde{A}_\Ss\right]^{-1} \transp{\widetilde{A}_\Ss} \diag^{-1}\left( R_\Ss \un{\bm p}\right) R_\Ss \\ & =  \left[\transp{\widetilde{A}_\Ss} \diag^{-1}\left( R_\Ss \un{\ggamma}^*(\un{\bm p})\right) \widetilde{A}_\Ss\right]^{-1} \transp{\widetilde{A}_\Ss} \diag^{-1}\left( R_\Ss \un{\bm p}\right) R_\Ss. \end{align*}


	\subsection*{Proof of Proposition \ref{prop:thetahat}}

	$(a)$ As Condition \ref{cond:Gamma:S:infinite} holds for $\Ss_\pp$, then Proposition \ref{prop:diff} guarantees that the function $\ttheta^*: \un{\Ps}(\Ss_\pp) \to \Theta$ in (\ref{eq:vartheta:*}) is continuous. Seeing that $\widehat{\ttheta}_n = \ttheta^*(\widehat{\un{\bm p}}_n)$ while $\ttheta_\pp = \ttheta^*(\un{\bm p})$, and that the operations $\vect$/$\vect^{-1}$ are continuous, the Continuous Mapping Theorem establishes that $\widehat{\ttheta}_n \overset{\text{P}}{\to} \ttheta_\pp$ ($\widehat{\ttheta}_n \overset{\text{as}}{\to} \ttheta_\pp$) as soon as $\widehat{\pp}_n \overset{\text{P}}{\to} \pp$ ($\widehat{\pp}_n \overset{\text{as}}{\to} \pp$) as $n \to \infty$.

	\ppn $(b)$ As $\widehat{\pp}_n$ approaches $\pp$ as $n \to \infty$, we can write the following stochastic expansion for $\ttheta^*: \un{\Ps}(\Ss_\pp) \to \Theta$ in (\ref{eq:vartheta:*}):
	\begin{equation} \ttheta^*(\widehat{\un{\bm p}}_n) = \ttheta^*(\un{\bm p}) + J_{\ttheta^*}(\un{\bm p})(\widehat{\un{\bm p}}_n - \un{\bm p}) + o_P(\|\widehat{\un{\bm p}}_n - \un{\bm p}\|) \label{eqn:stochexptheta} \end{equation}
	where $J_{\ttheta^*}(\un{\bm p})$ is the Jacobian matrix given in (\ref{eqn:jactheta}) with $\Ss= \Ss_\pp$. As $\widehat{\ttheta}_n = \ttheta^*(\widehat{\un{\bm p}}_n)$ and $\ttheta_\pp = \ttheta^*(\un{\bm p})$, the claim easily follows from Proposition \ref{prop:diff} and the Delta method \citep[Theorem 3.1]{VanderVaart00}.
	
	\ppn $(c)$ We immediately deduce from (\ref{eqn:stochexptheta}) that $\sqrt{n} \, \left(\widehat{\ttheta}_n - \ttheta_\pp\right) $ is an asymptotically normal vector if $\sqrt{n} \,(\widehat{\un{\bm p}}_n - \un{\bm p})$ is an asymptotically normal vector. Under condition (\ref{eqn:empprop}), $\sqrt{n} \, \left(\widehat{\ttheta}_n - \ttheta_\pp\right) $ is centred at $\bm 0$ and its variance-covariance matrix $\bm \Sigma_\ttheta$ is
	\begin{multline*}
		\bm \Sigma_\ttheta = \left[\transp{\widetilde{A}_{\Ss_\pp}} \diag^{-1}\left( R_{\Ss_\pp} \un{\ggamma}(\ttheta_\pp)\right) \widetilde{A}_{\Ss_\pp}\right]^{-1} \transp{\widetilde{A}_{\Ss_\pp}} \diag^{-1}\left( R_{\Ss_\pp} \un{\bm p}\right) R_{\Ss_\pp} \left[ \diag(\un{\bm p}) - \un{\bm p} \, \transp{\un{\bm p}} \right]  \\ \cdot \transp{R_{\Ss_\pp}} \diag^{-1}\left( R_{\Ss_\pp} \un{\bm p}\right) {\widetilde{A}_{\Ss_\pp}} \left[\transp{\widetilde{A}_{\Ss_\pp}} \diag^{-1}\left( R_{\Ss_\pp} \un{\ggamma}(\ttheta_\pp)\right) \widetilde{A}_{\Ss_\pp}\right]^{-1},   
	\end{multline*}
	where we have used (\ref{eqn:jactheta}) with $\Ss= \Ss_\pp$. The claim follows from seeing that $R_{\Ss_\pp} \diag(\un{\bm p})  \transp{R_{\Ss_\pp}} = \diag(R_{\Ss_\pp} \un{\bm p})$ and $$\diag^{-1}\left( R_{\Ss_\pp} \un{\bm p}\right) R_{\Ss_\pp}  \un{\bm p} \,  = \uno_{|{\Ss_\pp}|},$$
	where $\uno_{|{\Ss_\pp}|}$ is the vector of length $|{\Ss_\pp}|$ whose entries are all equal to 1; and finally $\transp{\widetilde{A}_{\Ss_\pp}} \uno_{|{\Ss_\pp}|} = \transp{{A}_{\Ss_\pp}} \transp{R}_{\Ss_\pp} \uno_{|{\Ss_\pp}|} = \transp{{A}_{\Ss_\pp}} \uno_{|\Rs|} = \bm 0$ as by definition (\ref{eq:DDelta}), all vectors of $\un{\DDelta}^\circ({\Ss_\pp})$ in (\ref{eq:DDelta:vect}) have entries summing to 0, and so it is the case for the columns $\un{\Delta}_k$ ($k = 1, \ldots, d_\circ$) of $\widetilde{A}_{\Ss_\pp}$.

	\section{The case $\Gamma(\Ss_\pp) = \emptyset$ ($d=2$)} \label{app:examples}
	
	Refer to Remark \ref {rem:mincond}. In terms of Lemma~\ref{lem:Gamma-nonempty}, case~$(a)$ occurs when $|\nu_1|/r_1+|\nu_2|/r_2\leq 1$ for every $\nu_1\times\nu_2\in N(\Ss_\pp)$, but equality holds for at least one such rectangle whose complement $\nu_1^c\times\nu_2^c$ does not belong to $N(\Ss_\pp)$ \citep[Theorem 6.1, case~$(b)(ii)$]{Geenens2020}. This is illustrated by the probability array 
	\[
	\pp = \begin{pmatrix} p_{11} & p_{12} & p_{13} \\ p_{21} & p_{22} & 0 \\ p_{31} & 0 & 0 \end{pmatrix}
	\]
	($r_1=r_2=3$), where all shown $p_{i_1i_2}$ are positive. For $\nu_1=\{3\}$ and $\nu_2=\{2,3\}$, the rectangle $\nu_1\times\nu_2$ contains only zeros and $|\nu_1|/r_1+|\nu_2|/r_2=1$, whereas its complementary rectangle $\{1,2\}\times\{1\}$ contains positive entries. Thus condition~(\ref{eqn:condbiv}) fails and $\Gamma(\Ss_\pp)=\emptyset$. Nevertheless, a uniform-margin array can be supported on a strict subset of $\Ss_\pp$. Indeed, the third row requires $\gamma_{31}=1/3$, forcing $\gamma_{11}=\gamma_{21}=0$; the second row then requires $\gamma_{22}=1/3$, forcing $\gamma_{12}=0$; and finally $\gamma_{13}=1/3$. Hence the only such array is
	\[
	\ggamma = \begin{pmatrix} 0 & 0 & 1/3 \\ 0 & 1/3 & 0 \\ 1/3 & 0 & 0 \end{pmatrix},
	\]
	whose support $\Ss'$ is strictly contained in $\Ss_\pp$, so that $\Gamma(\Ss')\neq\emptyset$. This is a critical support configuration: the $\Is$-projection of $\pp$ onto the generalised Birkhoff polytope  (\ref{eqn:gammadef}) exists but lies at one of its vertices, with support strictly smaller than $\Ss_\pp$. This support reduction recalls the distinction between `absolute' and `complete' association in the setting of \citet[Section 5.4]{Geenens2020}, where the distinction reflects constraints exclusively imposed by the prescribed margins.

	\ppn The following probability array illustrates case \((b)\):
	\[
	\pp = \begin{pmatrix} p_{11} & 0 & 0 & 0 \\ p_{21} & 0 & 0 & 0 \\ p_{31} & p_{32} & p_{33} & p_{34} \\ 0 & p_{42} & p_{43} & p_{44} \end{pmatrix}
	\]
	($r_1=r_2=4$), where all shown $p_{i_1 i_2}$'s are positive. The upper-right $2\times 3$ block of zeros violates condition~(\ref{eqn:condbiv}), since $2/4+3/4>1$. Concretely: the first two rows are supported entirely in the first column; under uniform row margins, $\gamma_{11}=\gamma_{21}=1/4$, so the first column must have mass at least $1/2$, which contradicts the uniform column margin of $1/4$. Hence no copula array of (\ref{eqn:gammadef}) admits a support equal to or contained in the support $\Ss_\pp$ of the above $\pp$.
	
	\ppn Now, recall that uniform margins are not mathematically forced upon us: they simply constitute a particularly natural choice for dependence representation as they are uninformative and, in the continuous setting, arise canonically through the probability integral transform. When globally uniform margins are incompatible with a discrete support, however, one may instead use other reference margins. For the above $\pp$, the support $\Ss_\pp$ contains 9 cells. Consider the uniform distribution $\qq$ on the admissible support, defined by
	\[
	q_{ij}=\frac19\indic{(i,j)\in \Ss_\pp}.
	\]
	Its row margins are $\bm q^{(1)}  = \left(\nicefrac19,\nicefrac19,\nicefrac49,\nicefrac39\right),$ whereas its column margins are $\bm q^{(2)} = \left(\nicefrac39,\nicefrac29,\nicefrac29,\nicefrac29\right)$. These margins therefore provide a natural support-adapted reference standard: they are induced by assigning equal weight to every admissible cell, rather than equal weight to every marginal level. In fact, $\qq$ is the unique maximum-entropy, and hence least informative, joint distribution among all probability distributions supported on $\Ss_\pp$. Its margins are therefore the margins induced by the least informative joint distribution on the prescribed support.
	
	\ppn From Lemma \ref{lem:multrepres} and Corollary \ref{cor:suppproj} in Section \ref{subsec:IprojFrech},  $\pp$ has a unique $\Is$-projection  $\boldsymbol\psi_{\pp} \doteq  \Is_{\Fs(\bm q^{(1)}, \bm q^{(2)})}(\pp)$ on the Fr\'echet class $\Fs(\bm q^{(1)},\bm q^{(2)})$, since  $\qq \in \Fs(\bm q^{(1)},\bm q^{(2)}) $ is such that $\Supp(\qq) = \Ss_\pp$, and that $\Is$-projection takes the form
	\[
	\boldsymbol\psi_{\pp} 
	=
	\begin{pmatrix}
		\nicefrac{1}{9}&0&0&0\\
		\nicefrac{1}{9}&0&0&0\\
		\nicefrac{1}{9}&\psi_{32}&\psi_{33}&\psi_{34}\\
		0&\psi_{42}&\psi_{43}&\psi_{44}
	\end{pmatrix},
	\]
	where $\psi_{3j}+\psi_{4j}=\frac29$ ($j=2,3,4$), $\psi_{i2}+\psi_{i3}+\psi_{i4} = \frac39$ ($i=3,4$) and $\psi_{ij}=\alpha_i\beta_jp_{ij}$ ($(i,j)\in \Ss_\pp$) for suitable positive scaling factors $\{\alpha_i\}$ and $\{\beta_j\}$. The last scaling condition identifies the unique member of $\Fs(\bm q^{(1)},\bm q^{(2)})$ having the same dependence structure as $\pp$. 
	
	\ppn Of course,  $\boldsymbol\psi_{\pp}  \notin \Gamma$, so it is not a copula array as such. We refer to it as a `{\it pseudo-copula array}', as it nevertheless retains the central purpose of a copula: $\boldsymbol\psi_{\pp} $ is a distribution-morphic representation of the dependence structure of $\pp$ isolated from its original margins. Consequently, although a representative with globally uniform margins need not exist, an always-feasible support-adapted pseudo-copula array can always be constructed.
	
	\ppn The support-adapted pseudo-copula array coincides with the usual uniform-margin copula array precisely when the uniform distribution on the support has uniform one-dimensional margins; equivalently,  when the number of positive entries is the same across all slices obtained by fixing any coordinate. This condition is trivially satisfied when the support is (hyper)rectangular. Another simple example when the above support-adapted pseudo-copula array coincides with the usual copula array is when $\Ss_\pp$ corresponds to the diagonal of a bivariate square array.
	
	\section{R script} \label{app:Rscript}
	
	In practice, the copula array $\ggamma_{\pp}$ is obtained as the output of the iterative proportional fitting (IPF) procedure, which is a well-established algorithm with efficient implementations available in many software packages. In our numerical work, we used the \texttt{mipfp} package in \textsf{R}, as noted in Remark \ref{rmk:IPF}. For illustration, the following \textsf{R} code computes the copula array of a bivariate probability array.
	
	\begin{lstlisting}[style=Rcode]
		require(mipfp)
		Pcop <- function(P){
			marg <- list(1,2)
			R <- nrow(P)
			C <- ncol(P)
			unimarg <- list(rep(1/R, R), rep(1/C, C))
			P0 <- Estimate(	seed = P,
			target.list = marg,
			target.data = unimarg)$p.hat
			P0
		}
	\end{lstlisting}

	\ppn The returned matrix \texttt{P0} is the copula array associated with the initial probability array (matrix) \texttt{P}, which consists of \texttt{R} rows and \texttt{C} columns. The function \texttt{Estimate} implements the IPF algorithm in the \texttt{mipfp} package and enforces discrete uniform marginal distributions ({\tt unimarg}) on the initial seed \texttt{P}.  Extending the script to a $d$-dimensional array is straightforward.
	
	\ppn The support of the input matrix \texttt{P} is assumed to satisfy Condition \ref{cond:Gamma:S}, otherwise the IPF procedure may fail to converge. That condition may be numerically checked beforehand; see Appendix \ref{app:checkCond31}.
	
	\section{Computational verification of Condition \ref{cond:Gamma:S}} \label{app:checkCond31}
	
	The following linear program, adapted from the binary version of \citet{Fontana26} to arbitrary marginal cardinalities, provides a computational check of Condition~\ref{cond:Gamma:S}. It maximises a common lower bound $t$ on the entries of a uniform-margin array supported within the prescribed set $\mathcal S\subseteq\mathcal R$. The decision variables are $t$ and $\{\gamma_{\boldsymbol i}:\boldsymbol i\in\mathcal R\}$:
	\[
	\text{maximise } t \qquad \text{subject to } \left\{\begin{array}{l c l}
		\gamma_{\boldsymbol i}=0 & \ \ & \boldsymbol i\in\Rs\setminus\mathcal S,\\
		\gamma_{\boldsymbol i}\geq t & \ \  & \boldsymbol i\in\mathcal S,\\
		t\geq0,\\
		\displaystyle\sum_{\boldsymbol i \in\mathcal S: i_\ell=j}\gamma_{\boldsymbol i}=\frac{1}{r_\ell} & \ \  & \ell\in[d],\quad j\in[r_\ell].
	\end{array}\right.
	\]
	Thus, the variables $\gamma_{\boldsymbol i}$ represent the entries of a candidate copula array, while $t$ is a common lower bound for all entries on $\mathcal S$. The entries outside $\mathcal S$ are fixed at zero. This is a standard linear program (all constraints and the objective are linear) that can be solved efficiently using readily available optimisation software.
	
	\ppn If the program is feasible and $t^\star$ denotes its optimal value, then
	\[
	\Gamma(\mathcal S)\neq\varnothing \quad\Longleftrightarrow\quad t^\star>0.
	\]
	Indeed, $t^\star>0$ means that there exists $\ggamma \in \Gamma$ satisfying $\gamma_{\boldsymbol i}\geq t^\star>0$ for every $\boldsymbol i\in\mathcal S$, and hence having support exactly $\mathcal S$ as $\gamma_{\boldsymbol i} = 0$ for every $\boldsymbol i\notin\mathcal S$. Conversely, if $\gamma\in\Gamma(\mathcal S)$, then $t=\min_{\boldsymbol i\in\mathcal S} \gamma_{\boldsymbol i}>0$ is feasible. If the program is feasible but $t^\star=0$, some $\ggamma \in \Gamma$ can be supported within $\mathcal S$, but none can have support exactly $\mathcal S$; if the program is infeasible, no $\ggamma \in \Gamma$ can be supported within $\mathcal S$.

	\section{Example \ref{ex:example2}} \label{app:ex2}
	
	For a $(4 \times 2 \times 2)$-array of support $\Ss_\pp$ as in (\ref{eqn:pex2}), the matrix $C_{\Ss_\pp}$ in (\ref{eq:linsystem}) is
	
	\[
	C_{\Ss_\pp} = \left(\begin{array}{*{16}{c}}  
		0 & 1 & 0 & 0 & 0 & 0 & 0 & 0 & 0 & 0 & 0 & 0 & 0 & 0 & 0 & 0 \\
		0 & 0 & 0 & 0 & 0 & 1 & 0 & 0 & 0 & 0 & 0 & 0 & 0 & 0 & 0 & 0 \\
		1 & 1 & 1 & 1 & 1 & 1 & 1 & 1 & 0 & 0 & 0 & 0 & 0 & 0 & 0 & 0 \\
		0 & 0 & 0 & 0 & 0 & 0 & 0 & 0 & 1 & 1 & 1 & 1 & 1 & 1 & 1 & 1 \\
		1 & 1 & 1 & 1 & 0 & 0 & 0 & 0 & 1 & 1 & 1 & 1 & 0 & 0 & 0 & 0 \\
		0 & 0 & 0 & 0 & 1 & 1 & 1 & 1 & 0 & 0 & 0 & 0 & 1 & 1 & 1 & 1 \\
		1 & 0 & 0 & 0 & 1 & 0 & 0 & 0 & 1 & 0 & 0 & 0 & 1 & 0 & 0 & 0 \\
		0 & 1 & 0 & 0 & 0 & 1 & 0 & 0 & 0 & 1 & 0 & 0 & 0 & 1 & 0 & 0 \\
		0 & 0 & 1 & 0 & 0 & 0 & 1 & 0 & 0 & 0 & 1 & 0 & 0 & 0 & 1 & 0 \\
		0 & 0 & 0 & 1 & 0 & 0 & 0 & 1 & 0 & 0 & 0 & 1 & 0 & 0 & 0 & 1
	\end{array} \right)
	\]
	The R procedure {\tt pracma::null} computed on this matrix returns the following matrix $A_{\Ss_\pp}$, whose $d_\circ = 8$ columns form a basis of the null space of $C_{\Ss_\pp}$:
	\[
	A_{\Ss_\pp} = \left( \begin{array}{*{8}{r}}  
		-0.4590 &  0.1005 &  0.0473 & -0.0418 & -0.1605 &  0.3990 &  0.3458 &  0.2567 \\
		0.0000 &  0.0000 &  0.0000 &  0.0000 &  0.0000 &  0.0000 &  0.0000 &  0.0000 \\
		0.2621 & -0.2955 & -0.1190 &  0.2176 &  0.4466 & -0.1109 &  0.0655 &  0.4021 \\
		-0.1980 & -0.3407 & -0.2871 & -0.4191 & -0.0842 & -0.2269 & -0.1733 & -0.3053 \\
		0.0639 & -0.1063 &  0.2149 &  0.3702 & -0.3323 & -0.5026 & -0.1813 & -0.0260 \\
		0.0000 &  0.0000 &  0.0000 &  0.0000 &  0.0000 &  0.0000 &  0.0000 &  0.0000 \\
		0.0131 &  0.4428 & -0.2971 & -0.0093 & -0.1131 &  0.3166 & -0.4233 & -0.1355 \\
		0.3179 &  0.1992 &  0.4410 & -0.1176 &  0.2435 &  0.1248 &  0.3666 & -0.1920 \\
		0.6713 &  0.0119 & -0.1664 & -0.2284 & -0.1791 &  0.1615 & -0.0167 & -0.0787 \\
		-0.0263 &  0.5089 & -0.0353 & -0.0642 &  0.0745 & -0.3903 &  0.0655 &  0.0366 \\
		-0.1639 & -0.0648 &  0.6728 & -0.1683 & -0.0922 &  0.0069 & -0.2556 & -0.0967 \\
		-0.0862 &  0.0797 & -0.1123 &  0.7042 & -0.0052 &  0.1608 & -0.0312 & -0.2147 \\
		-0.2762 & -0.0060 & -0.0958 & -0.1000 &  0.6719 & -0.0579 & -0.1477 & -0.1520 \\
		0.0263 & -0.5089 &  0.0353 &  0.0642 & -0.0745 &  0.3903 & -0.0655 & -0.0366 \\
		-0.1113 & -0.0826 & -0.2566 & -0.0400 & -0.2413 & -0.2126 &  0.6134 & -0.1699 \\
		-0.0337 &  0.0618 & -0.0417 & -0.1675 & -0.1542 & -0.0587 & -0.1622 &  0.7120
	\end{array} \right)
	\]
	The second and sixth rows of $A_{\Ss_\pp}$ are identically null, reflecting the location of the structural zeros in (\ref{eqn:pex2}).

\end{document}